\documentclass[11pt]{article}
\usepackage{amsmath}
\usepackage{amsthm}
\usepackage{enumerate}
\usepackage{bbold}
\usepackage{dsfont}
\usepackage[left=2.9cm,right=2.9cm,top=3.3cm,bottom=3.4cm,a4paper]{geometry}
\usepackage{amssymb}
\usepackage{color}
\usepackage[all]{xy}
\usepackage{xcolor}
\usepackage{titlesec}
\usepackage{abstract}
\usepackage{hyperref, theoremref}
\usepackage{lipsum}
\usepackage{here}
\usepackage[nottoc,notlot,notlof]{tocbibind}
\usepackage{graphicx} 
\usepackage{tikz} 
\usepackage{framed}

\numberwithin{equation}{section}
\allowdisplaybreaks[0]

\begin{document}


\title{Effective Mixing and Counting in Bruhat-Tits Trees}
\author{Sanghoon Kwon}
\date{\today}

\maketitle
\begin{abstract}
Let $\mathcal{T}$ be a locally finite tree, $\Gamma$ be a discrete subgroup of $\textrm{Aut}(\mathcal{T})$ and $\widetilde{F}$ be a $\Gamma$-invariant potential. Suppose that the length spectrum of $\Gamma$ is not arithmetic. In this case, we prove the exponential mixing property of the geodesic translation map $\phi\colon \Gamma\backslash S\mathcal{T}\to \Gamma\backslash S\mathcal{T}$ with respect to the measure $m_{\Gamma,F}^{\nu^-,\nu^+}$ under the assumption that $\Gamma$ is \emph{full} and $(\Gamma,\widetilde{F})$ has \emph{weighted spectral gap} property. We also obtain the effective formula for the number of $\Gamma$-orbits with weights in a Bruhat-Tits tree $\mathcal{T}$ of an algebraic group. 
\end{abstract}

\setcounter{tocdepth}{2}

\newtheorem{thm}{Theorem}[section]
\newtheorem{lem}[thm]{Lemma}
\newtheorem{prop}[thm]{Proposition}
\newtheorem{coro}[thm]{Corollary}

\def\BigRoman{\uppercase\expandafter{\romannumeral\number\count 255 }}
\def\Romannumeral{\afterassignment\BigRoman\count255=}

\theoremstyle{definition}
\newtheorem{conj}{Conjecture}[section]
\newtheorem{defn}{Definition}[section]
\newtheorem{exam}{Example}[section]

\theoremstyle{remark}
\newtheorem{remark}{Remark}[section]

\renewenvironment{proof}{\noindent\emph{Proof. }}{\hfill$\square$\bigskip}
\newcommand\numberthis{\addtocounter{equation}{1}\tag{\theequation}}
\renewcommand{\abstractnamefont}{\normalfont\large\bfseries}
\renewcommand{\theenumi}{\Alph{enumi}}

\bibliographystyle{alpha} 

\titleformat{\section}[block]{\large\scshape\filcenter}{\thesection.}{1em}{}

\titleformat{\subsection}[hang]{\bfseries}{\thesubsection.}{1em}{}

\section{Introduction}
Let $X$ be a geodesic metric space, $G<\textrm{Aut}(X)$ a closed subgroup of the isometry group $X$ and let $\{B_R\,\colon\,R>1\}$ be a family of compact subsets of $X$ whose volume goes to infinity as $R\to\infty$. We are interested in the following fundamental problem: given a discrete subgroup $\Gamma$ of $G$, what are the asymptotics and error rates of $|x\Gamma\cap B_R|$ for a point $x\in X$? 

If $\Gamma$ is a lattice subgroup of $G$, or equivalently if $\Gamma\backslash G$ admits a finite $G$-invariant measure, then this problem is well understood for various $G$. In this case, $|x\Gamma\cap B_R|$ is asymptotically proportional to the volume of $B_R$ with the suitable error term assuming that the boundaries of $B_R$ are sufficiently regular (\cite{EM},\cite{DRS},\cite{GN}). However, if $\Gamma$ is not a lattice subgroup, then this type of counting problem is little understood in general. 

When $G$ is the identity component of $SO(n,1)$, $(n\ge 2)$ and $H<G$ is either a horospherical subgroup or a symmetric subgroup, Oh and Shah computed an asymptotic formula of $|x\Gamma\cap B_R|$ for a family of compact subsets $\{B_R\subset G/H\}$ provided that the Bowen-Margulis measure on $T^1(\Gamma\backslash\mathbb{H}^n)$ and the $\Gamma$-skinning size of $x$ are finite (\cite{OS}). These conditions are satisfied when $\Gamma<G$ is a \emph{geometrically finite} subgroup, that is when the unit neighborhood of the convex core $\mathcal{C}_\Gamma$ has finite volume. In addition, Mohammadi and Oh computed an asymptotic formula with an error term via the effective asymptotic expansion of matrix coefficients of the complementary series (\cite{MO}). 

Our main focus in this paper is the case when $G$ is the $\mathbb{F}$-points of a semi-simple $\mathbb{F}$-rank one algebraic group over a \emph{non-Archimedean} local field $\mathbb{F}$. In this case, we can view $G$ as a closed subgroup of the automorphism group of a bi-regular tree $\cal{T}$ and hence $G$ acts on the space of geodesics $S\mathcal{T}$ in $\mathcal{T}$. We are interested in the formula which counts the number of points of a $\Gamma$-orbit in $\cal{T}$ when $\Gamma$ is a discrete subgroup of $G$ for which certain measure $m_{\Gamma,F}^{\nu^-,\nu^+}$ on the space of geodesics $\Gamma\backslash S\mathcal{T}$ is finite.

Let $\phi\colon \Gamma\backslash S\mathcal{T}\to\Gamma\backslash S\mathcal{T}$ be the geodesic translation map and $m_{\Gamma,F}^{\nu^-,\nu^+}$ be the measure on $\Gamma\backslash S\mathcal{T}$ associated to the potentials $\widetilde{F}^\pm\colon E\mathcal{T}\to\mathbb{R}$ and the normalized Patterson densities $\nu^\pm$. We will consider a Markov chain $Z_n$ with the countable state space $\mathcal{S}$, the transition probabilities $p_{ij}$ and the stationary distribution $\pi_j$ induced from the dynamical system $(\Gamma\backslash S\mathcal{T},\phi,m_{\Gamma,F}^{\nu^-,\nu^+})$ (see Section \ref{three}). To count the number of points of discrete orbits with weights in $\mathcal{T}$ effectively, the main ingredient we show is the exponential mixing property of geodesic 2-translation map $\phi^{\circ 2}$ with respect to the measure $m_{\Gamma,F}^{\nu^-,\nu^+}$. 

When the \emph{length spectrum} $L_\Gamma$ of the group $\Gamma$ (see Section \ref{two}) is equal to $k\mathbb{Z}$, the geodesic $k$-translation map $\phi^{\circ k}$ is mixing on $\Gamma\backslash S_o^k\mathcal{T}$ (\cite{BP},\cite{ro}). If we fix a reference state $s_0\in\mathcal{S}$, then we have a partition $(A_0,\cdots,A_{k-1})$ of $\mathcal{S}$ given by $A_m=\{s\in\mathcal{S}\colon p_{s_0s}^{(kn+m)}>0\textrm{ for some }n\in\mathbb{N}\}$. The exponential mixing property of $\phi^{\circ k}$ follows directly if the irreducible aperiodic $k$-step Markov chain $Z_{kn}$ on each subset $A_m\,(m=0,\cdots k-1)$ converges exponentially to the stationary distribution. To require such a good convergence property, we introduce the following \emph{WSG} property (\emph{weighted spectral gap}) for a pair of a full discrete group $\Gamma$ and a potential $\widetilde{F}$.

\begin{defn}\label{full} For a given discrete subgroup $\Gamma<\textrm{Aut}(\mathcal{T})$, let $\pi\colon\mathcal{T}\to\Gamma\backslash\mathcal{T}$ be the natural projection. We call $\Gamma_f=\{g\in\textrm{Aut}(\mathcal{T})\,|\,\pi\circ g=\pi\}$ the \emph{associated full subgroup} of $\Gamma$. A discrete subgroup $\Gamma<\textrm{Aut}(\mathcal{T})$ is called \emph{full} if $\Gamma=\Gamma_f$.
\end{defn}

\begin{defn} Let $\Gamma<\textrm{Aut}(\mathcal{T})$ be a full discrete subgroup and $\widetilde{F}\colon E\mathcal{T}\to\mathbb{R}$ be a $\Gamma$-invariant real valued function. Let $Z_{n}$ be the Markov chain with the data $(\mathcal{S},p_{ij},\pi_j)$ associated with $(\Gamma\backslash S\mathcal{T},\phi,m_{\Gamma,F}^{\nu^-,\nu^+})$ (see Section \ref{three}). Suppose that there is a function $t\colon \mathcal{S}\to\mathbb{R}_{\ge 0}$ given by $t(s_i)=t_i$, a finite subset $B\subset \mathcal{S}$ and a constant $0<\rho<1$ such that for every $s_i\in\mathcal{S}-B$, we have 
\begin{align}\label{WSG1}\sum_{s_j}p_{ij}t_jt_i^{-1}\le \rho.\end{align}
Then we say that $(\Gamma,\widetilde{F})$ has \emph{WSG} property with $(t,B,\rho)$.
If $\widetilde{F}$ is a constant, then we say that $\Gamma$ has \emph{WSG} property with $(t,B,\rho)$.
\end{defn}

 Now we can state our main theorem.

\begin{thm}\label{main} Let $\mathcal{T}$ be a locally finite uniform tree and let $\Gamma$ be a non-elementary full discrete subgroup of $\textrm{Aut}(\mathcal{T})$. Let $\widetilde{F}$ be a potential for $\Gamma$ such that $(\Gamma,\widetilde{F})$ has WSG property. If $|m_{\Gamma,F}^{\nu^-,\nu^+}|<\infty$ and $L_\Gamma=k\mathbb{Z}$, then for any $f,g\in C_c(\Gamma\backslash S_o^k\mathcal{T})$, as $n\to\infty$ we have
$$\left|\int_{\Gamma\backslash S_o^k\mathcal{T}} (f\circ \phi^{\circ kn})\cdot g\, dm_{\Gamma,F}^{\nu^-,\nu^+}-\int_{\Gamma\backslash S_o^k\mathcal{T}} f\,dm_{\Gamma,F}^{\nu^-,\nu^+}\int_{\Gamma\backslash S_o^k\mathcal{T}} g \,dm_{\Gamma,F}^{\nu^-,\nu^+}\right|=O(\theta^n),$$
for some constant $0<\theta<1$. For given $(\Gamma,\widetilde{F})$, the implied constant depends only on $f$ and $g$. \color{black}
\end{thm}

We remark here that if $(\Gamma,\widetilde{F})$ has WSG property with $(t,B,\rho)$, then $p_{ij}^{(n),B}\le t_it_j^{-1}\rho^n$ for each $s_i,s_j\in \mathcal{S}$ and this inequality is the key ingredient in the proof of Theorem \ref{main}.

Since every geometrically finite discrete subgroup $\Gamma<\textrm{Aut}(\mathcal{T})$ of a bi-regrular tree $\mathcal{T}_{r+1,s+1}$ satisfies WSG property (see Section \ref{four}), the above theorem implies that the geodesic $k$-translation map $\phi^{\circ k}\colon \Gamma\backslash S_o^k\mathcal{T}$ $\to\Gamma\backslash S_o^k\mathcal{T}$ is exponentially mixing for every geometrically finite group $\Gamma$. We also give an example of a group $\Gamma$ which is not geometrically finite but which has WSG property in Example \ref{example} (see Section \ref{four} for more examples and the detail).  Although there exists a Markov chain $(\mathcal{S},p_{ij},\pi_j)$ which does not satisfy the equation (\ref{WSG1}) (see Example \ref{counterexample}), we leave as a question whether there exists a $\Gamma<\textrm{Aut}(\mathcal{T})$ such that $|m_{\Gamma}^{\textrm{BM}}|<\infty$ without WSG property.

\begin{exam}\label{example} Let $q\ge 5$ be an arbitrary odd integer. Consider an edge-indexed ray of type $(2,q-1)$ (see Figure \ref{(2,q-1)}). If $\Gamma\backslash\backslash\mathcal{T}$ is a union of a finite graph of groups, finite such rays and finite funnels, then $\Gamma$ has WSG property. Therefore, by Theorem \ref{main} the geodesic $2$-translation map $\phi^{\circ 2}$ is exponentially mixing with respect to $m_\Gamma^{\textrm{BM}}$, which corresponds to the measure $m_{\Gamma,F}^{\nu^-,\nu^+}$ with the constant potential $\widetilde{F}$.

\end{exam}

\begin{figure}[H]\label{(2,q-1)}
\centering\includegraphics[width=0.8\linewidth]{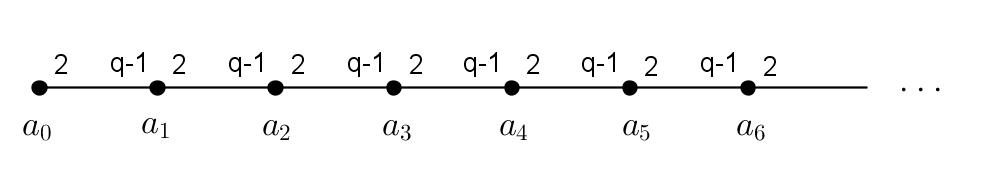}
\caption{Ray of type $(2,q-1)$}
\label{ray}
\end{figure}

\begin{exam}\label{counterexample}
This example shows the existence of a positive recurrent Markov chain which does not converge exponentially to the invariant distribution. Let $Z_n$ be the Markov chain on the state space $\mathcal{S}=\mathbb{Z}\cup\{\infty\}$ with the transition probabilities $p_{\infty n}=\beta_n$, $p_{nn}=\gamma_n$ and $p_{n \infty}=1-\gamma_n$. If $\gamma_n\to 1$ as $|n|\to\infty$, then the equation (\ref{WSG1}) does not hold for any $t\colon\mathcal{S}\to\mathbb{R}_{\ge 0}$. Nevertheless, if $\beta_n$ satisfies $$\sum_{n=-\infty}^{+\infty}\sum_{k=0}^{+\infty}\beta_n\gamma_n^k(1-\gamma_n)(k+2)<\infty,$$ then the invariant distribution $\pi_j=\lim_{k\to\infty}p_{ij}^{(k)}$ exists (see Subsection \ref{correlation}).
\begin{figure}[H]
\centering\includegraphics[width=0.7\linewidth]{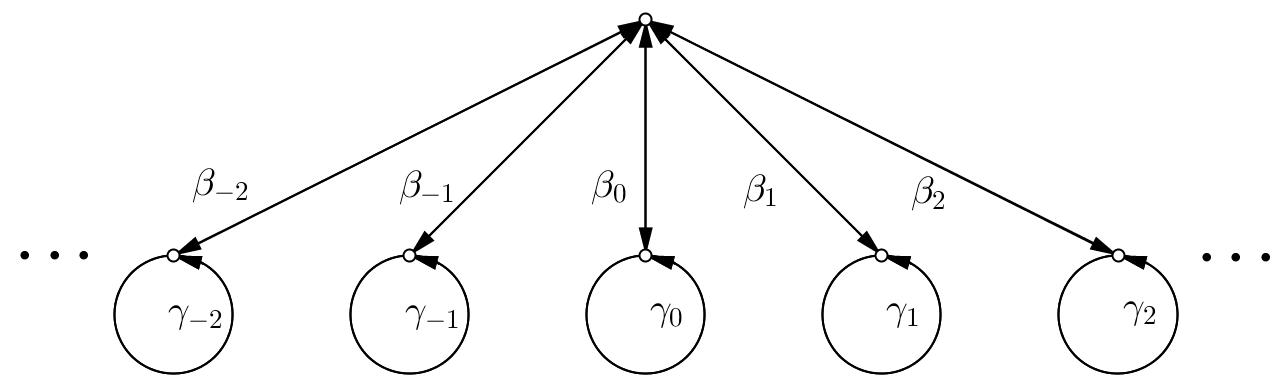}
\caption{Example of Markov chain which does not satisfy the equation (\ref{WSG1})}
\label{}
\end{figure}
\end{exam}
Now we consider an $\mathbb{F}$-rank one algebraic group $\underline{G}$ and its $\mathbb{F}$-rational points $G=\underline{G}(\mathbb{F})$. We will define a measure $\mathcal{M}_G^{(\Gamma)}$ on $G$ which depends on the group $\Gamma$ in Section \ref{five}. Due to the result of Roblin (\cite{ro}), for a family of compact subset $\{B_R\subset G\}$ which is \emph{well-rounded} (see Section \ref{five} for the definition), the number of points in the intersection of $B_R$ with a $\Gamma$-orbit is asymptotically $\mathcal{M}_G^{(\Gamma)}(B_R)$. From Theorem \ref{main}, we have the following corollary about the error rate of $|\Gamma g\cap B_R|$. 

\begin{coro} Let $\underline{G}$ be an $\mathbb{F}$-rank one semi-simple algebraic group over a local field $\mathbb{F}$, $G=\underline{G}(\mathbb{F})$ and $\mathcal{T}$ be the Bruhat-Tits tree of $G$. Let $\Gamma<G$ be a discrete subgroup with $|m_{\Gamma}^{\textrm{BM}}|<\infty$ with WSG property. Then for any well-rounded family $\{B_R\colon R>1\}$ of compact subsets of $G$, 
there is a constant $\eta>0$ for which we have
$$\vert\Gamma g\cap B_R\vert =\mathcal{M}_G^{(\Gamma)}(B_R)+O(\mathcal{M}_G^{(\Gamma)}(B_R)^{1-\eta})$$
as $R\to\infty$.
\end{coro}

More generally, we can count the number of $\Gamma$-orbits with weights corresponding to the given potential $\widetilde{F}\colon E\mathcal{T}\to\mathbb{R}$. Let
$$\mathcal{N}_x(n)=\sum_{\gamma\colon d_\mathcal{T}(\gamma x,x)\le n}e^{\int_x^{\gamma x}\widetilde{F}}$$ be the weighted counting function of paths of length at most $n$ in $\mathcal{T}$ which start at $x$ and end at $\gamma x$. Then using the equidistribution of the skinning measure as in \cite{PPS}, we have the following corollary.
\begin{coro} Let $x$ be a degree $q^d+1$ vertex of the Bruhat-Tits tree $\mathcal{T}$ of $G$ and suppose that $(\Gamma,\widetilde{F})$ is a pair of discrete subgroup $\Gamma<G$ and a potential $\widetilde{F}$ for $\Gamma$ with WSG property.
Then as $n\to\infty$, we have
$$\mathcal{N}_x(2n)=\frac{e^{2\delta_{\Gamma,F}}\|\nu^-_x\|\|\nu^+_x\||\Gamma_x|}{(e^{2\delta_{\Gamma,F}}-1)\|m_{\Gamma,F}^{\nu^-,\nu^+}\|}e^{2n\delta_{\Gamma,F}}+O(e^{(2\delta_{\Gamma,F}-\kappa)n})$$
for some $\kappa>0$. 
\end{coro}

This paper is organized as follows. In Section 2, we recall preliminaries including the notion of edge-indexed graph and graph of groups, limit points and conformal densities on the space $\Gamma\backslash \textrm{Aut}(\mathcal{T})$, and the Bruhat-Tits tree of rank 1 algebraic groups. In Section 3, we translate the dynamical system $(\Gamma\backslash S\mathcal{T},\phi,m_{\Gamma,F}^{\nu^-,\nu^+})$ into a countable Markov chain $(\mathcal{S},p_{ij},\pi_j)$ following \cite{BM}. In Section 4, we define WSG property and prove our main theorem. We also give some examples of discrete groups with WSG property. Finally, we prove an effective formula of the number of $\Gamma$ orbits in the Bruhat-Tits tree $\mathcal{T}$ of $G$ in Section 5.

\bigskip
\noindent\emph{Acknowledgement.} We would like to express our gratitude to Seonhee Lim for her patient help and suggestions. We also thank Fr\'ed\'eric Paulin for his interest and giving numerous comments. The author is supported in part by Subsequent Generation Scholarship of Basic Research Program in Seoul National University and NRF-2013R1A1A2011942.

\section{Notations and Preliminaries}\label{two}
\subsection{Graph of groups and edge-indexed graph}\label{edge-indexed graph}
For a graph $A$, we denote by $VA$ the set of vertices of $A$ and by $EA$ the set of \emph{oriented} edges of $A$. For $e\in EA$, let $\overline{e}\in EA$ be the opposite edge of $e$ and let $\partial_0e$ and $\partial_1e$ be the initial vertex and the terminal vertex of $e$, respectively.

By a \emph{graph of groups} $\mathbf{A}=(A,\mathcal{A})$ we mean a connected graph $A$ together with groups $\mathcal{A}_a$ $(a\in VA)$, $\mathcal{A}_e=\mathcal{A}_{\overline{e}}$ $(e\in EA)$, and monomorphisms $\alpha_e\colon\mathcal{A}_e\to\mathcal{A}_{\partial_1e}$ $(e\in EA)$. An \emph{isomorphism} between two graph of groups $\mathbf{A}=(A,\mathcal{A})$ and $\mathbf{A'}=(A',\mathcal{A'})$ is an isomomorphism $\phi\colon A\to A'$ between two underlying graphs together with the set of isomomorphisms $\phi_a\colon \mathcal{A}_a\to \mathcal{A'}_{\phi(a)}$ and $\phi_e\colon \mathcal{A}_e\to \mathcal{A'}_{\phi(e)}$ satisfying the following property: for each $e\in EA$, there is an element $h_e\in \mathcal{A'}_{\phi(\partial_1e))}$ such that 
$$\phi_{\partial_1e}(\alpha_e(g))=h_e\cdot(\alpha'_{\phi(e)}(\phi_e(g)))\cdot h_e^{-1}$$
for all $g\in\mathcal{A}_e$ (\cite{B}).

For each $e\in EA$, we put $i(e)=[\mathcal{A}_a\colon\alpha_e\mathcal{A}_e]$. We then call $I(\mathbf{A})=(A,i)$ the \emph{edge-indexed graph} of $\mathbf{A}$. We will assume that all indices $i(e)$ are finite and define $q(e)=i(\overline{e})/i(e)\in\mathbb{Q}_{>0}^{\times}$ for each $e\in EA$. Let $\pi(A)$ be the path group of $A$. Since $q(\overline{e})=q(e)^{-1}$, it follows that $q\colon \pi (A)\to\mathbb{Q}_{>0}^{\times}$ \color{black}defines a homomorphism.

Given a locally finite tree $\mathcal{T}$, let $\Gamma$ be a closed subgroup of $\textrm{Aut}(\mathcal{T})$, which is the group of all isometries of $\mathcal{T}$ acting without inversions. Then the quotient graph $\Gamma\backslash\mathcal{T}$ has a natural structure of graph of groups, which we will denote by $\Gamma\backslash\backslash \mathcal{T}$, as follows. For each $v\in V(\Gamma\backslash\mathcal{T})$ and $e\in E(\Gamma\backslash\mathcal{T})$, choose any corresponding vertex $\widetilde{v}\in V\mathcal{T}$ and edge $\widetilde{e}\in E\mathcal{T}$. Let $\overline{\widetilde{e}}=\widetilde{\overline{e}}$ and fix an element $\gamma_e\in\Gamma$ which satisfies $\gamma_e\widetilde{\partial_1(e)}=\partial_1(\widetilde{e})$. Define $\mathcal{A}_v$ and $\mathcal{A}_e$ be the stabilizer of $\widetilde{v}$ and $\widetilde{e}$ in $\Gamma$, respectively, and let $\alpha_e\colon\mathcal{A}_e\to\mathcal{A}_{\partial_1e}$ be the monomorphism $h\mapsto \gamma_e^{-1}h\gamma_e$. Then the \emph{quotient graph of groups} $\Gamma\backslash\backslash\mathcal{T}=(\Gamma\backslash\mathcal{T},\mathcal{A})$ does not depend on the choice of $\widetilde{v},\widetilde{e}$ and $\gamma_e$, up to isomorphism of graph of groups. Let $I(\Gamma\backslash\backslash\mathcal{T})=(\Gamma\backslash\mathcal{T},i)$ be the edge-indexed graph of $\Gamma\backslash\backslash\mathcal{T}$. Conversely, if $(A,\mathcal{A})$ is a graph of groups and $(A,i)$ is the corresponding edge-indexed graph, then fixing a basepoint $a_0\in VA$, the universal covering tree $\mathcal{T}=\widetilde{(A,a_0)}$ and the natural projection $\pi\colon \mathcal{T}\to A$ depend only on the edge indexed graph $(A,i)$ (\cite{Se}).

\subsection{Geometrically finite discrete groups acting on locally finite trees\color{black}}
For a given locally finite simplicial tree $\mathcal{T}$, we define a metric $d$ \color{black}on $V\mathcal{T}$ so that $d(u,v)$ is the number of edges of the segment between $u$ and $v$. Let us denote by $S\cal{T}$ the space of all bi-infinite geodesics in $\cal{T}$, i.e., the set of isometries $\xi\colon\mathbb{Z}\to V\mathcal{T}$. We call an isometry $r\colon \mathbb{N}\cup\{0\}\to V\mathcal{T}$ a \emph{ray}.
Let $\widetilde{\phi}\colon S\mathcal{T}\to S\mathcal{T}$ be the \emph{forward geodesic translation} given by $(\widetilde{\phi}\xi)(t)=\xi(t+1)$ and $\widetilde{\iota}\colon S\mathcal{T}\to S\mathcal{T}$ be the \emph{geodesic inversion} given by $(\widetilde{\iota}\xi)(t)=\xi(-t+1)$. Then $\widetilde{\phi}$ and $\widetilde{\iota}$ commute with the action of $\psi\in \textrm{Aut}(\mathcal{T})$. The group $\mathbb{Z}$ acts on $S\mathcal{T}$ by $n\mapsto \widetilde{\phi}^{\circ n}$. 

Let $\widetilde{\pi}\colon S\mathcal{T}\to \Gamma\backslash S\mathcal{T}$ be the natural projection map and let $\phi\colon\Gamma\backslash S\mathcal{T}\to\Gamma\backslash S\mathcal{T}$ be the induced forward geodesic translation given by $\phi\circ\widetilde{\pi}=\widetilde{\pi}\circ\widetilde{\phi}$ and let $\iota\colon\Gamma\backslash S\mathcal{T}\to \Gamma\backslash S\mathcal{T}$ be the induced geodesic inversion defined similarly.



The \emph{Busemann cocycle} is the map $\beta\colon\partial_\infty\mathcal{T}\times\mathcal{T}\times\mathcal{T}\to\mathbb{Z}$, given by 
$$(\omega,x,y)\mapsto\beta_\omega(x,y)=d(x,v)-d(y,v)$$
for some (hence any) $v\in r_x\cap r_y$ where $r_x$ and $r_y$ are the rays starting from $x$ and $y$, respectively, and converging to $\omega$. 
The Busemann cocycle satisfies the following properties: $\beta_\omega(x,y)+\beta_\omega(y,z)=\beta(x,z)$, $\beta_\omega(x,y)=-\beta_\omega(y,x)$ and $\beta_{g\cdot\omega}(g\cdot x,g\cdot y)=\beta_\omega(x,y)$ for all $g\in G$, $x,y,z\in\mathcal{T}$ and $\omega\in\partial_\infty\mathcal{T}$.

Given $\xi\in S\mathcal{T}$, let $\xi^+\in\partial_\infty\mathcal{T}$ be the positive end and $\xi^-\in\partial_\infty\mathcal{T}$ be the negative end. A \emph{contracting horosphere} based at $\xi^+$ is the subset $\mathcal{H}_\xi^+=\{\eta\in S\mathcal{T}\,|\,\eta^+=\xi^+\textrm{ and }\beta_{\xi^+}(\xi(0),\eta(0))=0\}$. An \emph{expanding horosphere} based at $\xi^-$ is the subset $\mathcal{H}_{\xi}^-=\{\eta\in S\mathcal{T}\,|\,\eta^-=\xi^-\textrm{ and }\beta_{\xi^-}(\xi(0),\eta(0))=0\}$.

Let $\Gamma$ be a discrete subgroup of $\textrm{Aut}(\mathcal{T})$. The \emph{limit set} $\Lambda_\Gamma$ of $\Gamma$ is the set of accumulation points of a $\Gamma$-orbit in $\cal{T}$. By the discreteness of $\Gamma$, we have $\Lambda_\Gamma\subseteq\partial_\infty\cal{T}$. The \emph{convex hull} $C\Lambda_\Gamma$ of $\Gamma$ is the smallest convex subset of $\overline{\cal{T}}=\mathcal{T}\cup\partial_\infty\cal{T}$ containing $\Lambda_\Gamma$. \par 
A point $\omega\in\Lambda_\Gamma$ is called a \emph{conical limit point} if there is a sequence $(\gamma_n)_{n\in\mathbb{N}}$ of $\Gamma$ such that for any point $x\in\cal{T}$, any ray $c$ which converges to $\omega$ and any $n\in\mathbb{N}$, we have $d(\gamma_nx,c)\le C$ for some $C>0$.
A point $\omega\in\Lambda_\Gamma$ is called a \emph{horocyclic limit point} if there is a sequence $(\gamma_n)_{n\in\mathbb{N}}$ of $\Gamma$ such that for any point $x\in\cal{T}$, we have $\displaystyle \lim_{n\to\infty}\beta_{\omega}(\gamma_nx,x)=\infty$.

\begin{defn}\label{cuspidal} A point $\omega\in\partial_\infty\mathcal{T}$ is called a $\Gamma$-\emph{parabolic point} if the stabilizer $\Gamma_\omega$ fixes no point in $\partial_\infty\mathcal{T}$ other than $\omega$ and fixes no vertex of $\mathcal{T}$. It is called $\Gamma$-\emph{cuspidal} if, further, for any ray with vertex sequence $x_0,x_1,x_2,x_3,\cdots$ toward $\omega$, we have $\Gamma_{x_n}\le \Gamma_\omega$ for any sufficiently large $n>0$. 
\end{defn}

A point $\omega\in\Lambda_\Gamma$ is called a \emph{bounded parabolic point} if the stabilizer $\Gamma_\omega$ of $\omega$ acts properly discontinuously and cocompactly on $\Lambda_\Gamma\backslash\{\omega\}$.

\begin{prop}[\cite{P}, Theorem 1.1]\label{Paulin} Let $\Gamma$ be a discrete subgroup of $\textrm{Aut}(\mathcal{T})$ and let us denote the minimal $\Gamma$-invariant subtree of $\cal{T}$ by ${\mathcal{T}}_{min}$. Then the followings are equivalent:\par\color{black}

\begin{enumerate}[(a)]\setlength\itemsep{-\parsep} 

\item Every limit point of $\Gamma$ is either a conical limit point or a bounded parabolic point.
\item Every limit point of $\Gamma$ is either a horocyclic limit point or a bounded parabolic point. 
\item The quotient graph of groups $\Gamma\backslash\backslash {\mathcal{T}}_{min}$ is a union of a finite graph of finite groups and a finite number of $\Gamma$-cuspidal rays of groups.
\end{enumerate}
\end{prop}

\color{black}
If the above equivalent conditions hold, then we say the subgroup $\Gamma$ is \emph{geometrically finite}. If, furthermore, $\Gamma\backslash\backslash\mathcal{T}_{\textrm{min}}$ is a finite graph of finite groups, then we say $\Gamma$ is \emph{convex cocompact}. Note that there is a lattice $\Gamma<\textrm{Aut}(\mathcal{T})$ which is \emph{not} geometrically finite. We also remark here that a discrete subgroup $\Gamma$ of Aut$(\mathbb{H}^n)$ is geometrically finite if one of the equivalent conditions (a) and (b) in the above proposition holds, or equivalently, the unit neighborhood of the convex core $\mathcal{C}_\Gamma=\Gamma\backslash C\Lambda_\Gamma$ has finite volume.

A \emph{funnel} $F$ is a subtree of $\mathcal{T}$ with exactly one vertex of degree 1 such that $\mathcal{T}-F$ is connected. Imitating the proof for hyperbolic plane case in \cite{Bo} and using (c) of the above theorem, we have the following proposition about the structure of quotient graph of groups $\Gamma\backslash\backslash\mathcal{T}$ when $\Gamma$ is geometrically finite.

\begin{prop}\label{split} If $\Gamma$ is geometrically finite, then there are a finite graph of groups $D$, finite $\Gamma$-cuspidal rays $C_1,\cdots,C_k$ and finite funnels $F_1,\cdots,F_l$ so that 
\begin{enumerate}[(1)]\setlength\itemsep{-\parsep} 
\item $VA=VD\cup VC_1\cup\cdots\cup VC_k\cup VF_1\cup\cdots\cup VF_l.$
\item $|VD\cap VC_j|=|VD\cap VF_m|=1$ and $VC_j\cap VF_m=\phi$.
\item If $a_{j,0}\in VD\cap VC_j$ and $e_{j,1}\in EC_j$ with $\partial_1e_{j,1}=a_{j,0}$, then $i(e_{j,1})=1$. 
\end{enumerate}
\end{prop}

Let $L_\Gamma$ be the subgroup of $\mathbb{Z}$ generated by the translation length $l(\gamma)=\min_{v\in V\mathcal{T}}\{d(v,\gamma\cdot v)\}$ of every element $\gamma\in\Gamma$. We call $L_\Gamma$ by the \emph{length spectrum} of $\Gamma$. If $\mathcal{T}$ has no proper non-empty $\Gamma$-invariant subtree and has no vertices of degree 2, then either $L_\Gamma=\mathbb{Z}$ or $L_\Gamma=2\mathbb{Z}$ (\cite{BP}). If $L_\Gamma=\mathbb{Z}$, then we say the length spectrum of $\Gamma$ is \emph{not arithmetic}. This property is required to prove that the geodesic translation map is mixing (see Theorem \hyperref[Roblin]{2.3}).

\subsection{Patterson-Sullivan theory with potential $\widetilde{F}$ for groups acting on trees}\label{density}

The theory of \emph{Patterson-Sullivan densities} was developed by Patterson (for $\mathbb{H}^2$) and Sullivan (for $\mathbb{H}^n,n\ge 3)$. It was generalized to groups acting on CAT$(-1)$ spaces by Burger-Mozes (\cite{BM}). 
See also \cite{CP} for the case of universal covers of finite simplicial graphs.
We review in this subsection the Patterson-Sullivan theory together with a non-zero potential $\widetilde{F}$, following \cite{PPS} and \cite{BPP}. 

A \emph{potential} $\widetilde{F}$ for a discrete group $\Gamma<\textrm{Aut}(\mathcal{T})$ is a $\Gamma$-invariant continuous function $\widetilde{F}\colon E\mathcal{T}\to\mathbb{R}$. We denote by $F$ the induced function $F\colon \Gamma\backslash E\mathcal{T}\to\mathbb{R}$ on the quotient. For all $x,y\in V\mathcal{T}$ with $d(x,y)=n$, we define $$\int_x^y\widetilde{F}=\sum_{i=0}^{n-1}\widetilde{F}(e_i)$$
where $\xi\in S\mathcal{T}$ is any bi-infinite geodesic with $\xi(0)=x$ and $\xi(n)=y$ and $e_i$ is the edge satisfying $\partial_0e_i=\xi(i)$ and $\partial_1e_i=\xi(i+1)$. This definition does not depend on the choice of $\xi\in S\mathcal{T}$.

\begin{remark} In general, a potential is a function defined on the unit tangent bundle $T^1\widetilde{M}$ of a negatively curved manifold $\widetilde{M}$ or the space of germs $T^1X$ of a CAT$(-1)$ space $X$. For simplicial trees $\mathcal{T}$, the unit tangent vector bundle (or the space of germs) is defined as the quotient space $S\mathcal{T}/\sim$ where $\xi\sim \xi'$ if and only if $\xi(0)=\xi'(0)$ and $\xi(1)=\xi'(1)$. Thus, we can canonically identify $S\mathcal{T}/\sim$ with $E\mathcal{T}$.
\end{remark}

The \emph{critical exponent} of $(\Gamma,F)$ is the element $\delta_{\Gamma,F}\in [-\infty,+\infty]$ defined by
$$\delta_{\Gamma,F}=\underset{n\to +\infty}{\limsup}\frac{1}{n}\log\sum_{\gamma\in\Gamma, d(x,\gamma y)=n}e^{\int_x^{\gamma y}\widetilde{F}}$$
and the \emph{Poincar\'e series} of $(\Gamma,F)$ is the map $Q=Q_{\Gamma,F,x,y}\colon \mathbb{R}\to [0,\infty]$ given by
$$Q\colon s\mapsto \sum_{\gamma\in \Gamma}e^{\int_x^{\gamma y} (\widetilde{F}-s)}.$$
If $F=0$, then $\delta_{\Gamma,0}$ and $Q_{\Gamma,0}$ are the usual critical exponent $\delta_\Gamma$ and Poincar\'e series $\sum_{\gamma\in\Gamma} e^{-sd(x,\gamma y)}$ of $\Gamma$, respectively.

Let $\widetilde{F}^+=\widetilde{F}$ and $\widetilde{F}^-=\widetilde{F}\circ\widetilde{\iota}$. We denote by $F^{\pm}\colon \Gamma\backslash E\mathcal{T}\to\mathbb{R}$ their induced maps. In the rest of paper, we assume that $\delta=\delta_{\Gamma,F^+}=\delta_{\Gamma,F^-}$ is finite. Generalizing the Busemann cocycle, the authors in \cite{PPS} defined the \emph{Gibbs cocycle} associated with the group $\Gamma$ and the potential $\widetilde{F}$ as the map $C_F\colon \partial_\infty\mathcal{T}\times\mathcal{T}\times\mathcal{T}\to\mathbb{R}$ given by
$$(\omega,x,y)\mapsto C_{F,\omega}(x,y)=\int_y^v \widetilde{F}-\int_x^v \widetilde{F}$$
for some (hence any) $v\in r_x\cap r_y$. We call $C_{F^-\delta}$ a \emph{normalized Gibbs cocycle} associated with $\Gamma$ and $\widetilde{F}$. 

\begin{defn}[Patterson density for $(\Gamma,\widetilde{F})$ (\cite{PPS})]
A \emph{Patterson density} of dimension $\alpha$ for the pair $(\Gamma,\widetilde{F})$ is a family of finite nonzero positive Borel measures $\{\mu_x\}_{x\in V\mathcal{T}}$ on $\partial_\infty\mathcal{T}$ such that for every $\gamma\in\Gamma$, for all $x,y\in\mathcal{T}$ and for every $\omega\in\partial_\infty\mathcal{T}$, we have 
$$\gamma_*\mu_x=\mu_{\gamma\cdot x}\qquad\textrm{and}\qquad \frac{d\mu_x}{d\mu_y}(\omega)=e^{-C_{F-\alpha,\omega}(x,y)}.$$
In particular, a $\delta_{\Gamma,F}$-dimensional Patterson density for $(\Gamma,\widetilde{F})$ is called \emph{normalized}.
\end{defn}

There is at least one normalized Patterson density for the pair $(\Gamma,\widetilde{F})$. Moreover, if $Q_{\Gamma,F}(\delta)=\infty$, then the normalized Patterson density is unique up to a multiplicative constant (\cite{PPS}, \cite{BPP}).


Now let $\{\nu_x^\pm\}$ be the \emph{normalized} Patterson densities for the pairs $(\Gamma,\widetilde{F}^\pm)$, respectively, and let $C^\pm$ be the \emph{normalized} Gibbs cocycle for $(\Gamma,\widetilde{F}^\pm)$. Fixing $o\in\cal{T}$, the map $\xi\mapsto (\xi^+,\xi^-,s)$ gives a homeomorphism between $S\cal{T}$ and $((\partial_\infty\mathcal{T}\times \partial_\infty\mathcal{T}) \backslash \Delta)\times \mathbb{Z}$, where $\xi(s)$ is the closest point to $o$ on the geodesic line $\xi$. Let us define the measure $\widetilde{m}_{F}^{\nu^-,\nu^+}$ on $S\cal{T}$ associated to $\{\nu_x^-\}$ and $\{\nu_x^+\}$ by
$$d\widetilde{m}_{F}^{\nu^-,\nu^+}(\xi)=e^{C^-_{\xi^-}(o,\xi_0)}e^{C^+_{\xi^+}(o,\xi_0)}d\nu_o^-(\xi^-)d\nu_o^+(\xi^+)ds,$$
where $ds$ is the counting measure on $\mathbb{Z}$.
It follows from the $\Gamma$-invariant conformal property of $\{\nu_x^\pm\}$ that the definition of $\widetilde{m}_F$ is independent of the choice of $o\in\cal{T}$ and that $\widetilde{m}_{F}^{\nu^-,\nu^+}$ is left $\Gamma$-invariant. Hence it induces a Radon measure $m_{\Gamma,F}^{\nu^-,\nu^+}$ of the quotient space $\Gamma\backslash S\mathcal{T}$. This definition is also invariant under the geodesic translation map $\phi_\Gamma$ on $\Gamma\backslash S\mathcal{T}$.
Now we may lift the measure $m_{\Gamma,F}^{\nu^-,\nu^+}$ to $\Gamma\backslash G$ for arbitrary closed subgroup $G$ of $\textrm{Aut}(\cal{T})$ via $M=\textrm{Stab}_G(\xi)$-invariant extension: for $\widetilde{f}\in C_c(\Gamma\backslash G)$, define $$m_{\Gamma,F}^{\nu^-,\nu^+}(\widetilde{f})=\int_{\eta\in\Gamma\backslash S\mathcal{T}}\int_{g\in M} \widetilde{f}(\eta g)dgdm_{\Gamma,F}^{\nu^-,\nu^+}.$$

When $F$ is constant, the normalized Patterson density is equal to the usual $\delta_{\Gamma}$-dimensional \emph{Patterson-Sullivan} density and the measure $m_{\Gamma,F}^{\nu^-,\nu^+}$ coincides with the \emph{Bowen-Margulis} measure $m_\Gamma^{\textrm{BM}}$ given by 
$$dm_\Gamma^{\textrm{BM}}=e^{\delta_\Gamma\beta_{\xi^+}(o,\xi_0)}e^{\delta_\Gamma \beta_{\xi^-}(o,\xi_0)}d\nu_o(\xi^+)d\nu_o(\xi^-)ds.$$ Recall that the length spectrum of $\Gamma$ is said to be not arithmetic when $L_\Gamma=\mathbb{Z}$ for $\Gamma<\textrm{Aut}(\mathcal{T})$. The following mixing property of the geodesic flow is due to \cite{ro} and \cite{BPP}.

\begin{thm} Let $\Gamma<\textrm{Aut}(\mathcal{T})$ be a discrete subgroup whose length spectrum is not arithmetic. Let $\widetilde{F}$ be a given potential for $\Gamma$ and let $\nu^\pm$ be the normalized Patterson densities for $(\Gamma,\widetilde{F}^\pm)$. Suppose that $|m_{\Gamma,F}^{\nu^-,\nu^+}|<\infty$. Then, for all compactly supported continuous functions $f,g\colon \Gamma\backslash S\mathcal{T}\to\mathbb{R}$ we have
$$\int_{\Gamma\backslash S\mathcal{T}} f\cdot(g\circ \phi_\Gamma^{\circ n}) dm_{\Gamma,F}^{\nu^-,\,\nu^+}\longrightarrow \frac{1}{|m_{\Gamma,F}^{\nu^-,\nu^+}|}\int_{\Gamma\backslash S\mathcal{T}} f dm_{\Gamma,F}^{\nu^-,\nu^+}\int_{\Gamma\backslash S\mathcal{T}} g dm_{\Gamma,F}^{\nu^-,\nu^+}.$$
\end{thm}

If $L_\Gamma=k\mathbb{Z}$, then let $S_o^k \mathcal{T}$ be the subset $\{\xi\in S\mathcal{T}\,|\,d(\xi(0),o)\in k\mathbb{Z}\}$ of $S\mathcal{T}$. This is invariant under $\Gamma$ and $\phi^{\circ k}$, so we may restrict $m_{\Gamma,F}^{\nu^-,\nu^+}$ to $\Gamma\backslash S_{o}^k\mathcal{T}$. In this case, the dynamical system $(\Gamma\backslash S_o^k\mathcal{T},\phi^{\circ k},m_{\Gamma,F}^{\nu^-,\nu^+})$ is mixing, for any choice of $o\in V\mathcal{T}$ (cf. \cite{BP}). 

\subsection{Bruhat-Tits tree of rank one algebraic groups}
Let $\mathbb{F}$ be a non-Archimedean local field with the valuation $v\colon\mathbb{F}\to\mathbb{Z}$. If $G=\underline{G}(\mathbb{F})$ is an $\mathbb{F}$-rational points of an arbitrary $\mathbb{F}$-rank one algebraic group over $\mathbb{F}$ whose residue field is $\mathbb{F}_q$, then the Bruhat-Tits tree $\cal{T}$ of $G$ is a $(q^d+1,q^{d'}+1)$-biregular tree where $d$ and $d'$ are integers attached to the two vertices of the relative local Dynkin diagram for $H$. Indeed, if $B$ is the minimal parabolic subgroup of $G$, and $P_1$ and $P_2$ are the two maximal proper parabolic subgroups of $G$, then $q^d+1=[P_1\colon B]$ and $q^{d'}+1=[P_2\colon B]$. By the classical result of the classification, up to strict isogeny we have 9 absolutely simple rank one algebraic groups. If $G$ is not simply connected, then $G$ may act with inversions. In this case, we can replace $G$ by a subgroup of index 2 which acts without inversions (see \cite{Tits}). From now on, we assume that $G$ acts on $\cal{T}$ without inversions; hence, $d_{\mathcal{T}}(x,g\cdot x)$ is always even for every $g\in G$ and every vertex $x$ of $\cal{T}$.

We now introduce some subgroups of $G$. Let $N^+$ and $N^-$ be the set of elements of $G$ which stabilize the contracting horosphere $\mathcal{H}_\xi^+(0)$ and expanding horosphere $\mathcal{H}_\xi^-(0)$, respectively. We can also choose a maximal $\mathbb{F}$-split torus $S$ of $\underline{G}$ which fixes $\xi^+$ and $\xi^-$ in $\partial_\infty\mathcal{T}$ pointwise. Let $Z=Z_{\underline{G}}(S)_{\mathbb{F}}$ be the centralizer of $S$ in $\underline{G}$. For each $j\in\mathbb{Z}$, there is a unique element $z$ in $Z$ such that $z\cdot H_\xi^+(0)=H_\xi^+(2j)$; we denote such $z$ by $a_{2j}$ (see \cite{Tits}). It follows that $Z=\{a_{2j}\,\colon j\in\mathbb{Z}\}$. Define $Z^+=\{a_{2j}\,\colon j\ge 0\}$.
Finally, let $P^+=ZN^+$ and $P^-=ZN^-$.

\color{black}


\section{Changing the dynamical system of geodesic translation map into Markov chain}\label{three}

In this section, we introduce the Markov chain $(\mathcal{S},p_{ij},\pi_j)$ with countable states which is associated to the dynamical system $(\Gamma\backslash S\mathcal{T},\phi,m_{\Gamma,F}^{\nu^-,\nu^+})$. We follow the idea appeared in \cite{BM} and take some notions there with a slight modification. The novelty of this section is that we interpret the exponential mixing properties of the dynamical system $(\Gamma\backslash S\mathcal{T},\phi,m_{\Gamma,F}^{\nu^-,\nu^+})$ in terms of a countable Markov chain.

\subsection{Countable Markov shift $X_{(A,i)}$}
Let $\mathcal{T}$ be a locally finite tree and $S\mathcal{T}$ be the set of all bi-infinite geodesics in $\mathcal{T}$. Let $\widetilde{\phi}\colon S\mathcal{T}\to S\mathcal{T}$ be the \emph{forward geodesic translation} given by $(\widetilde{\phi}\xi)(t)=\xi(t+1)$ and $\widetilde{\iota}\colon S\mathcal{T}\to S\mathcal{T}$ be the \emph{geodesic inversion} given by $(\widetilde{\iota}\xi)(t)=\xi(-t+1)$. The group $\mathbb{Z}$ acts on $S\mathcal{T}$ by $n\mapsto \widetilde{\phi}^{\circ n}$. Then $\widetilde{\phi}$ and $\widetilde{\iota}$ commute with the action of $\psi\in \textrm{Aut}(\mathcal{T})$.

Let $\widetilde{\pi}\colon S\mathcal{T}\to \Gamma\backslash S\mathcal{T}$ be the natural projection map and let $\phi\colon\Gamma\backslash S\mathcal{T}\to\Gamma\backslash S\mathcal{T}$ be the induced forward geodesic translation given by $\phi\circ\widetilde{\pi}=\widetilde{\pi}\circ\widetilde{\phi}$ and let $\iota\colon\Gamma\backslash S\mathcal{T}\to \Gamma\backslash S\mathcal{T}$ be the induced geodesic inversion defined similarly. We will assume that $\Gamma$ is a full discrete subgroup of $\textrm{Aut}(\mathcal{T})$, i.e., a discrete subgroup of $\textrm{Aut}(\mathcal{T})$ which is equal to its associated full subgroup $\Gamma_f=\{g\in\textrm{Aut}(\mathcal{T})\,|\,\pi\circ g=\pi\}$ where $\pi\colon\mathcal{T}\to \Gamma\backslash\mathcal{T}$ is the canonical projection (see Definition \ref{full}).

For a given edge indexed graph $(A,i)$, consider the following subset $$X_{(A,i)}=\{x=(e_j)_{j\in\mathbb{Z}}\,|\,\partial_0e_{j+1}=\partial_1e_j\textrm{ and if } e_{j+1}=\overline{e_j},\textrm{ then }i^A(e_j)>1\}$$ of $(EA)^{\mathbb{Z}}$. The family of \emph{cylinders} 
$$[e_0,\cdots,e_{n-1}]:=\{x\in X_{(A,i)}\colon x_i=e_i,i=0,\cdots,n-1\}$$ is a basis of open sets for a topology on $X$. Let $\sigma\colon X_{(A,i)}\to X_{(A,i)}$ be given by $\sigma(x)_i:=x_{i+1}$. Then $(S\mathcal{T},\widetilde{\phi})$ is conjugate to $(X_\mathcal{T},\sigma)$ (Consider $\mathcal{T}$ as $(\mathcal{T},i^0)$ with $i^0(e)=1,\forall e\in E\mathcal{T}$). If we denote by $(A,i)$ the edge-indexed graph associated with $\Gamma\backslash\backslash\mathcal{T}$, then we also have a bijection $\Phi$ between two spaces $(\Gamma\backslash S\mathcal{T},\phi)$ and $( X_{(A,i)},\sigma)$ so that the following diagram commute.
\begin{equation}
\begin{aligned}\label{1}
\xymatrix{ \Gamma\backslash S\mathcal{T} \ar[r]^{\quad\Phi\quad} \ar[d]_{\phi} & X_{(A,i)} \ar[d]_{\sigma} \\ 
\Gamma\backslash S\mathcal{T} \ar[r]^{\quad\Phi\quad} & X_{(A,i)}}
\end{aligned}
\end{equation}
Indeed, let $\pi\colon \mathcal{T}\to\Gamma\backslash\mathcal{T}$ be the canonical projection and let $\xi\colon\mathbb{Z}\to\mathcal{T}$ be an element of $S\mathcal{T}$. Let $\xi_j$ be the edge satisfying $\partial_0\xi_j=\xi(j)$ and $\partial_1\xi_j=\xi(j+1)$. Define $\widetilde{\Phi}(\xi)=(\pi(\xi_j))_{j\in\mathbb{Z}}$. Then $\widetilde{\Phi}\colon S\mathcal{T}\to X_{(A,i)}$ is $\Gamma$-invariant, so it induces the map $\Phi\colon\Gamma\backslash S\mathcal{T}\to X_{(A,i)}$. Given $e_1,e_2\in E\mathcal{T}$ satisfying $\partial_0e_1=\partial_0e_2$ and $\pi(e_1)=\pi(e_2)$, there is $\gamma\in\textrm{Aut}(\mathcal{T})$ such that $\gamma e_1=e_2,\pi\circ\gamma=\pi$, and $\gamma$ acts trivially on the connected component of $\mathcal{T}-\{e_1,e_2\}$ containing $\partial_0e_1$. Therefore, $\Phi$ is a $\phi$-equivariant homeomorphism (cf. \cite{BM}).

\subsection{Invariant measures on $(X_{(A,i)},\sigma)$}
\label{measure}
Let $\textbf{A}=(A,\mathcal{A})$ be a graph of groups and $(A,i)$ be the corresponding edge-indexed graph. In this subsection, we will get a \emph{Markov measure} $\lambda_F$ on the dynamical system $(X_{(A,i)},\sigma)$ corresponding to $m_{\Gamma,F}^{\nu^-,\nu^+}$. In other words, we require
\begin{equation}
\begin{aligned}\label{1.5}
\sum_{e_j\colon\partial_0e_k=\partial_1e_j}\lambda_F([e_j,e_k]) &=\lambda_F([e_k]), \\
\sum_{e_k\colon\partial_0e_k=\partial_1ej}\lambda_F([e_j,e_k]) &=\lambda_F([e_j])
\end{aligned}
\end{equation} 
and
\begin{align}\label{2}
\int_{\Gamma\backslash S\mathcal{T}}f \,dm_{\Gamma,F}^{\nu^-,\nu^+}= \int_{X_{(A,i)}}\Phi^*(f)\,d\lambda_F
\end{align}
where $\Phi^*(f)[(x_i)_{i\in\mathbb{Z}}]=f[\Phi^{-1}((x_i)_{i\in\mathbb{Z}})].$

Fixing a basepoint $a_0\in VA$, the universal covering $\mathcal{T}\simeq\widetilde{(\textbf{A},a_0)}$ is a locally finite tree. Let $\pi\colon\mathcal{T}\to\Gamma\backslash\mathcal{T}$ be the natural projection. We define the invariant measure $\lambda_F$ on $(X_{(A,i)},\sigma)$ as follows. For $e\in E\mathcal{T}$, we denote by $$\mathcal{O}(e)=\{\omega\in\partial_\infty\mathcal{T}\,|\, \exists \xi\in S\mathcal{T}\textrm{ such that }\xi(0)=\partial_0e,\xi(1)=\partial_1e\textrm{ and }\xi^+=\omega\},$$ the \emph{shadow} of an edge $e$. For each cylinder $[e_0,\cdots,e_{n-1}]$, let $\xi_0,\cdots,\xi_{n-1}\in E\mathcal{T}$ be the edges of a path in $\mathcal{T}$ satisfying $\pi(\xi_j)=e_j$ and let $$\Gamma_{\xi_0,\cdots,\xi_{n-1}}:=\textrm{Stab}_{\Gamma}(\xi_0)\cap\cdots\cap\textrm{Stab}_\Gamma(\xi_{n-1}).$$ 
Let $\nu^\pm$ be the normalized Patterson density for $(\Gamma,\widetilde{F}^\pm)$ constructed in Subsection \ref{density}. Now we define 
$$\lambda_F([e_0,\cdots,e_{n-1}])=\frac{\nu^-_{\partial_0\xi_0}(\mathcal{O}(\overline{\xi_0}))\nu^+_{\partial_1\xi_{n-1}}(\mathcal{O}(\xi_{n-1}))}{|\Gamma_{\xi_0,\cdots,\xi_{n-1}}|}e^{\int_{\partial_0\xi_0}^{\partial_1\xi_{n-1}}\widetilde{F}-\delta}.$$
Since $|\Gamma_{\xi_0,\cdots,\xi_{n-1}}|$ depends only on $e_0,\cdots,e_{n-1}$ and $\nu^\pm$ are $\Gamma$-invariant, the measure $\lambda_F$ is well-defined. Moreover, the equation (\ref{2}) holds since
\begin{align*}
& \lambda_F([e_0,\cdots,e_{n-1}])=\frac{\nu^-_{\partial_0\xi_0}(\mathcal{O}(\overline{\xi_0}))\nu^+_{\partial_1\xi_{n-1}}(\mathcal{O}(\xi_{n-1}))}{|\Gamma_{\xi_0,\cdots,\xi_{n-1}}|}e^{\int_{\partial_0\xi_0}^{\partial_1\xi_{n-1}}\widetilde{F}-\delta} \\
= \,& \frac{1}{|\Gamma_{\xi_0,\cdots,\xi_{n-1}}|} \int_{\xi^+\in\partial_\infty\mathcal{T}}\mathbb{1}_{\mathcal{O}(\xi_{n-1})} e^{-C_{\xi^+}(\partial_1\xi_{n-1},\partial_0\xi_0)} \int_{\xi^-\in\partial_\infty\mathcal{T}} \mathbb{1}_{\mathcal{O}(\overline{\xi_0})} d\nu^-_{\partial_0\xi_0}(\xi^-)d\nu^+_{\partial_1\xi_{n-1}}(\xi^+) \\
= \,& 
\int_{\Gamma\backslash S\mathcal{T}}\mathbb{1}_{\Phi^{-1}[e_0,\cdots,e_{n-1}]}d\nu^-_{\partial_0\xi_0}d\nu^+_{\partial_0\xi_0} =m_{\Gamma,F}^{\nu^-,\nu^+}(\Phi^{-1}[e_0,\cdots,e_{n-1}]).
\end{align*}
It remains us to prove the Markov property (\ref{1.5}) of $\lambda_F$.

\begin{prop}[Markov property]\label{3} Let us define $$p_{e_j,e_k}:= \frac{\lambda_F([e_j,e_k])}{\lambda_F([e_j])}.$$ Then, the measure $\lambda_F$ and $p_{e_j,e_k}$ satisfy the following:
$$\sum_{e_j\colon\partial_0e_k=\partial_1e_j}\lambda_F([e_j])p_{e_j,e_k}=\lambda_F([e_k])\quad\textrm{and}\quad \sum_{e_k\colon\partial_0e_k=\partial_1e_j}p_{e_j,e_k}=1.$$ 
\end{prop}
\begin{proof}
Let $\Gamma\backslash\backslash\mathcal{T}=(A,\mathcal{A})$ and $e_k\in EA$ with $\partial_0e_k=\partial_1e_j$. If $e_k=\overline{e_j}$, then $|\Gamma_{\xi_k}|=(i^A({e_j})-1)|\Gamma_{\xi_j,\xi_k}|$ and if $e_k\ne\overline{e_j}$, then $|\Gamma_{\xi_k}|=i^{A}(\overline{e_k})|\Gamma_{\xi_j,\xi_k}|$. In each case, let $$\rho_{e_j,e_k}=\frac{|\Gamma_{\xi_j}|}{|\Gamma_{\xi_j,\xi_k}|}.$$

Since 
$\nu_{\partial_1\xi_j}^+(\mathcal{O}_{\xi_k})=\nu_{\partial_1\xi_k}^+ (\mathcal{O}_{\xi_k})e^{\widetilde{F}(\xi_k)-\delta}$ and $\nu^-_{\partial_0\xi_k}(\mathcal{O}_{\overline{\xi_j}})=\nu^-_{\partial_0\xi_j}(\mathcal{O}_{\overline{\xi_j}})e^{\widetilde{F}(\xi_j)-\delta}$,
we have $$p_{e_j,e_k}=\frac{|\Gamma_{\xi_j}|\nu^+_{\partial_1\xi_j}(\mathcal{O}(\xi_k))}{|\Gamma_{\xi_j,\xi_k}|\nu^+_{\partial_1\xi_j}(\mathcal{O}(\xi_j))}\quad\textrm{and}\quad \frac{\lambda_F([e_j,e_k])}{\lambda_F([e_k])}=\frac{|\Gamma_{\xi_k}|\nu^-_{\partial_0\xi_k}(\mathcal{O}(\overline{\xi_j}))}{|\Gamma_{\xi_j,\xi_k}|\nu^-_{\partial_0\xi_k}(\mathcal{O}(\overline{\xi_k}))}.$$

By the definition of edge-indexed graph $(A,i)$ (see Subsection \ref{edge-indexed graph}) and the finite (even countable) additivity of the normalized Patterson density $\nu$, we have $$
\nu^+_{\partial_1\xi_j}(\mathcal{O}(\xi_j))=\sum_{e_k\colon\partial_0e_k=\partial_1e_j} \rho_{e_j,e_k}\nu^+_{\partial_1\xi_j}(\mathcal{O}(\xi_k))$$ and 
$$\nu^-_{\partial_0\xi_k}(\mathcal{O}(\xi_k))=\sum_{e_j\colon\partial_1e_j=\partial_0e_k} \rho_{e_k,e_j}\nu^-_{\partial_0\xi_k}(\mathcal{O}(\xi_j))
$$ which completes the proof.
\end{proof}

The equations in (\ref{1.5}) follow from the above proposition. 
Therefore, it follows that the measure $\lambda_F$ is a Markov measure and two dynamical systems $$(\Gamma\backslash S\mathcal{T},\phi,m_{\Gamma,F}^{\nu^-,\nu^+}) \textrm{ \ and \ }(X_{(A,i)},\sigma,\lambda_F)$$ are isomorphic by (\ref{1}) and (\ref{2}).

\subsection{Decay of correlation in terms of Markov chain $(\mathcal{S},p_{ij},\pi_j)$}
\label{correlation}
We consider a Markov chain $Z_n$ with phase space $\mathcal{S}=\{s_1,s_2,\cdots\}$ and transition probabilities $$p_{ij}=p_{s_is_j}=P\{Z_{n+1}=s_j\,|\,Z_n=s_i\}.$$ For a subset $B\subset \mathcal{S}$ of alphabets, let 
\begin{equation}
\begin{aligned}\label{4}
f_{ij}^{(n),B} &=P\{Z_1\notin B\cup s_j,\cdots, Z_{n-1}\notin B\cup s_j, Z_n=s_j\,|Z_0=s_i\}, \\
p_{ij}^{(n),B} &=P\{Z_1\notin B,\cdots, Z_{n-1}\notin B, Z_n=s_j\,|Z_0=s_i\}.
\end{aligned}
\end{equation}
We also denote by $f_{ij}^{(n)}=f_{ij}^{(n),\phi}$ and $p_{ij}^{(n)}=p_{ij}^{(n),\phi}$. By convention, $f_{ij}^{(0)}=0$ and $p_{ij}^{(0)}=\delta_{ij}$. Meanwhile, we have the following convolution relation \begin{align}\label{conv} p_{s_is_j}^{(n),B}=\sum_{r=1}^{n}f_{s_is_i}^{(r),B}p_{s_is_j}^{(n-r),B}.
\end{align}
Suppose that the Markov chain $Z_n$ is irreducible, i.e., for any $s_i,s_j\in\mathcal{S}$, there exists $n>0$ such that $p_{ij}^{(n)}>0$. We say $\pi_j$ is the \emph{stationary} (or \emph{invariant}) distribution if it satisfies $\pi_j=\sum_{i\in\mathcal{S}}\pi_i p_{ij}$. If it exists, then we say the Markov chain $(\mathcal{S},p_{ij})$ is \emph{recurrent} (and otherwise we say \emph{transient}). When the Markov chain $(\mathcal{S},p_{ij},\pi_j)$ is recurrent, it is called \emph{positive recurrent} if $\sum_{n=1}^{\infty}nf_{jj}^{(n)}<\infty$ (and otherwise called \emph{null recurrent}). When $(\mathcal{S},p_{ij},\pi_j)$ is positive recurrent, we have $$\pi_j=\frac{1}{\sum_{n=1}^{\infty}nf_{jj}^{(n)}}.$$
A \emph{period} of a state $s_i\in\mathcal{S}$ is defined by $k=\textrm{gcd}\{n\colon p_{ii}^{(n)}>0\}$. An irreducible Markov chain is called \emph{aperiodic} if for some (and hence every) state $s_i\in \mathcal{S}$, the period is 1. 
If the positive recurrent Markov chain $Z_n$ is irreducible and aperiodic, then $\pi_j=\underset{n\to\infty}{\lim}p_{ij}^{(n)}$ and $\pi_j$ does not depend on the choice of $i\in\mathcal{S}$ (\cite{MT}).

We are interested in the Markov chain $(\mathcal{S},p_{ij},\pi_j)$ for \begin{align}\label{MCdef}\mathcal{S}=E(\Gamma\backslash\backslash\mathcal{T}),\quad p_{ij}=p_{s_is_j}=\frac{\lambda_F([s_i,s_j])}{\lambda_F([s_i])}, \quad\textrm{and}\quad \pi_j=\lambda_F([s_j]).\end{align}
If the length spectrum of $\Gamma$ is not arithmetic, then $(\Gamma\backslash S\mathcal{T}, \phi,m_{\Gamma,F}^{\nu^-,\nu^+})$ is mixing and hence $(\mathcal{S},p_{ij},\pi_j)$ is an irreducible aperiodic Markov chain.

Now we can deduce the exponential decay of correlation from the information of the speed of convergence to stationary distribution.

\begin{prop}\label{correlation} Suppose there are constants $C_{s_is_j}>0$ and $0<\theta<1$ such that $|p_{s_is_j}^{(n)}-\pi_{s_j}|\le C_{s_is_j}\theta^n$. Then 
$$|\textrm{Cov}(f,g)|=\left|\int (f\circ \sigma^{\circ n})\cdot g d\lambda_F-\int fd\lambda_F\int g d\lambda_F\right|\le C(f,g)\theta^n.$$ 
\end{prop}
\begin{proof} 
Since every measurable set can be approximated by finite disjoint unions of cylinder sets, we only need to show that $\left|\lambda_F([\underline{a}]\cap\sigma^{-n}[\underline{b}])-\lambda_F([\underline{a}])\lambda_F([\underline{b}])\right|$ converges to 0 exponentially, for all cylinder sets $[\underline{a}]=[a_0,\cdots,a_{k-1}]$ and $[\underline{b}]=[b_0,\cdots,b_{l-1}]$.
Now,
\begin{align*}
\displaybreak[0] \left|\textrm{Cov}(\mathbb{1}_{[\underline{a}]},\mathbb{1}_{[\underline{b}]})\right| 
=\,&\left|\lambda_F([\underline{a}]\cap\sigma^{ -n}[\underline{b}])-\lambda_F([\underline{a}])\lambda_F([\underline{b}])\right|\\ \displaybreak[0]
=\,&\left|\lambda_F\left(\bigcup_{\underline{c}\in W_{n-k}}[\underline{a},\underline{c},\underline{b}]\right) -\lambda_F([\underline{a}])\lambda_F([\underline{b}])\right|\\ \displaybreak[0]
=\,&\left|\lambda_F([\underline{a}])\left(\sum_{\underline{c}\in W_{n-k}}p_{a_{k-1}c_0}\cdots p_{c_{n-k-1}b_0}\cdot\frac{\lambda_F([\underline{b}])}{\pi_{b_0}}\right) -\lambda_F([\underline{a}])\lambda_F([\underline{b}])\right|\\ \displaybreak[0]
=\,&\lambda_F([\underline{a}])\lambda_F([\underline{b}])\left|\frac{p_{a_{k-1}b_0}^{(n-k)}-\pi_{b_0}}{\pi_{b_0}}\right| 
\le \|\mathbb{1}_{[\underline{a}]}\|\|\mathbb{1}_{[\underline{b}]}\|\frac{C_{a_{k-1}b_0}}{\pi_{b_0}\theta^k}\theta^n.
\end{align*}
This completes the proof.
\end{proof}


\section{Exponential decay of the correlation function}\label{four}

Let $\mathcal{T}$ be a locally finite tree. Let $(\Gamma,\widetilde{F})$ be a pair of full discrete subgroup $\Gamma<\textrm{Aut}(\mathcal{T})$ and a potential $\widetilde{F}$ for $\Gamma$ with $|m_{\Gamma,F}^{\nu^-,\nu^+}|<\infty$. Assume that $L_\Gamma=\mathbb{Z}/k\mathbb{Z}$ and let $S_o^k\mathcal{T}$ be a subset $\{\xi\in S\mathcal{T}\,|\,d(\xi_0,o)\in k\mathbb{Z}\}$ of $S\mathcal{T}$ for $o\in V\mathcal{T}$. 
Our goal in this section is to prove the exponential decay of the correlation function 
$$\left|\textrm{Cov}(f,g)\right| = \,\left|\int_{\Gamma\backslash S_o^k\mathcal{T}} (f\circ \sigma^{\circ n})\cdot g dm_{\Gamma,F}^{\nu^-,\nu^+}-\int_{\Gamma\backslash S_o^k\mathcal{T}} fdm_{\Gamma,F}^{\nu^-,\nu^+}\int_{\Gamma\backslash S_o^k\mathcal{T}} g dm_{\Gamma,F}^{\nu^-,\nu^+}\right|$$ for $f,g\in C_c(\Gamma\backslash S_o^k\mathcal{T})$ when $\Gamma<\textrm{Aut}(\mathcal{T})$ has \textrm{WSG} property (see Definition \ref{WSG}). By Proposition \ref{correlation}, it suffices to show that $|p_{s_is_j}^{(n)}-\pi_{s_j}|\le C_{s_is_j}\theta^n$ for some $C_{s_is_j}>0$ and $0<\theta<1$. Inspired by L.-S. Young's observation about the relation between the speed of convergence to equilibrium states and the recurrent time to a compact set (\cite{Y}), we investigate the speed of decay of the probability $p_{s_is_j}^{(n),B}$ for a finite set $B\subset E(\Gamma\backslash\mathcal{T})$ (see (\ref{4})). 

\subsection{Exponential convergence to the stationary distribution}

\begin{defn}[WSG property]\label{WSG} Let $\Gamma<\textrm{Aut}(\mathcal{T})$ be a full discrete subgroup and $\widetilde{F}\colon E\mathcal{T}\to\mathbb{R}$ be a potential for $\Gamma$ such that $|m_{\Gamma,F}^{\nu^-,\nu^+}|<\infty$. Let $Z_{n}$ be the Markov chain with the data $(\mathcal{S},p_{ij},\pi_j)$ associated with $(\Gamma\backslash S\mathcal{T},\phi,m_{\Gamma,F}^{\nu^-,\nu^+})$ (see (\ref{MCdef})). Suppose that there is a function $t\colon \mathcal{S}\to\mathbb{R}_{\ge 0}$ given by $t(s_i)=t_i$, a finite subset $B\subset \mathcal{S}$ and a constant $0<\rho<1$ such that for $s_i\in\mathcal{S}-B$, we have 
\begin{align}\label{WSG}\sum_{s_j}p_{ij}t_jt_i^{-1}\le \rho.\end{align}
Then we say $(\Gamma,\widetilde{F})$ has \emph{WSG} (\emph{weighted spectral gap}) property with $(t,B,\rho)$.
\end{defn}

\begin{lem}\label{lemma} Suppose $(\Gamma,\widetilde{F})$ has WSG property with $t\colon\mathcal{S}\to\mathbb{R}_{\ge 0}$, $B\subset \mathcal{S}$ and $0<\rho <1$. Then, we obtain $p_{ij}^{(n),B}\le t_it_j^{-1}\rho^n$.
In particular, $$p_{i,B}^{(n),B}=P\{Z_1\notin B,\cdots, Z_{n-1}\notin B,Z_n\in B\,|\,Z_0=s_i\}\le Mt_i\rho^n$$
for $M=\max\{t_j^{-1}\,|\,s_j\in B\}$.
\end{lem}
\begin{proof} Let $t_i=t(s_i)$ and $t_j=t(s_j)$. Then
\begin{align*}
p_{ij}^{(n),B} &=\sum_{r_k\in \mathcal{S}-B}p_{s_ir_1}p_{r_1r_2}\cdots p_{r_{n-1}s_j}=\frac{t_i}{t_j}\sum_{r_k\in \mathcal{S}-B}\left(p_{s_ir_1}\frac{t(r_1)}{t_i}\right)\cdots \left(p_{r_{n-1}s_j}\frac{t_j}{t(r_{n-1})}\right) \\
&\le \frac{t_i}{t_j}\left(\sum_{r_1} p_{s_ir_1}\frac{t(r_1)}{t_i}\right)\left(\max_{r_1}\sum_{r_2} p_{r_1r_2}\frac{t(r_2)}{t(r_1)}\right)\cdots\left(\max_{r_{n-1}} p_{r_{n-1}j}\frac{t_j}{t(r_{n-1})}\right) \\
&\le t_it_j^{-1}\rho^n.
\end{align*}
The second statement follows directly from this inequality.
\end{proof}

The following two propositions are main ingredients of our proof of exponential mixing property. 
\begin{prop}\label{key1}
Let $(\mathcal{S},p_{ij})$ be an irreducible aperiodic countable Markov chain. If there exists a state $b_0\in\mathcal{S}$ which satisfies $p_{s_is_j}^{(n),\{b_0\}}\le C_{s_i,s_j}\theta^n$ for some $0<\theta<1$ , then it converges to the stationary distribution exponentially, i.e., $|p_{s_is_j}^{(n)}-\pi_{s_j}|<\widehat{C}_{s_i,s_j}\widehat{\theta}^n$ for some $0<\widehat{\theta}<1$ with the constants $\widehat{C}_{s_i,s_j}$ depending only on $s_i$ and $s_j$.
\end{prop}

\begin{proof}
Let $F_{s_is_j}(z)=\sum_{n=0}^{\infty}f_{ij}^{(n)}z^n$ and $P_{s_is_j}(z)=\sum_{n=0}^{\infty}p_{ij}^{(n)}z^n$ be the generating functions. 
By the assumption, we have $f_{s_ib_0}^{(n)}=p_{s_ib_0}^{(n),\{b_0\}}\le C_{s_i}\theta^{n}$. Thus, power series $F_{s_ib_0}(z)$ is analytic in the disk $C_{\theta^{-1}}=\{z\in\mathbb{C}\,|\,|z|<\theta^{-1}\}$. Moreover, from the convolution relation (\ref{conv}), the equality $$P_{b_0b_0}(z)=\frac{1}{1-F_{b_0b_0}(z)}$$ holds and $P_{b_0b_0}(z)$ has a unique simple pole at the point $z=1$ in the disk $C_R$ for some $1<R\le\theta^{-1}$. Since $$\lim_{z\to 1}\frac{z-1}{1-F_{b_0b_0}(z)}=\frac{1}{-\sum_{n=1}^{\infty} nf_{b_0b_0}^{(n)}}=-\pi_{b_0},$$ it follows that $P_{b_0b_0}(z)-\frac{\pi_{b_0}}{1-z}=\underset{n=0}{\overset{\infty}{\sum}}\left(p_{b_0b_0}^{(n)}-\pi_{b_0}\right)z^n$ is analytic for $|z|<R$ and hence there is a constant $c>0$ such that $$|p_{b_0b_0}^{(n)}-\pi_{b_0}|<c\theta_1^n$$ for some $\theta_1\le R^{-1}<1$.

Now using the equality 
$$p_{s_ib_0}^{(n)}=f_{s_ib_0}^{(n)}+\sum_{r=1}^{n-1}f_{s_i,b_0}^{(r)}p_{b_0b_0}^{(n-r)}$$
and the triangle inequality, we have $$\quad |p_{s_ib_0}^{(n)}-\pi_0|\le\left|\sum_{r=1}^{n}f_{s_ib_0}^{(r)}(p_{b_0b_0}^{(n-r)}-\pi_{b_0})\right|+\pi_{b_0}\sum_{r=n+1}^{\infty}f_{s_ib_0}^{(r)}.$$
This implies the existence of constant $C_{s_i}'>0$ for which
$$|p_{s_ib_0}^{(n)}-\pi_{b_0}|<C_{s_i}'\theta_2^n$$ with some $R^{-1}<\theta_2<1$.

Finally, since the Markov chain $(\mathcal{S},p_{ij},\pi_j)$ is irreducible and aperiodic, we have $\pi_{s_j}=\pi_{b_0}\displaystyle\sum_{n=0}^{\infty}p_{b_0,s_j}^{(n),\{b_0\}}$ and $$p_{s_is_j}^{(n)}=p_{s_i,s_j}^{(n),\{b_0\}}+\sum_{r=0}^{n-1}p_{s_ib_0}^{(r)} p_{b_0s_j}^{(n-r),\{b_0\}}.$$ Hence, we obtain
$$|p_{s_is_j}^{(n)}-\pi_{s_j}|\le p_{s_is_j}^{(n),\{b_0\}}+\left|\sum_{r=1}^{n-1}(p_{s_ib_0}^{(r)}-\pi_{b_0})p_{b_0s_j}^{(n-r),\{b_0\}}\right|+\pi_{b_0}\sum_{r=n}^{\infty}p_{b_0,s_j}^{(r),\{b_0\}}$$ which implies that there is a constant $C(s_i,s_j)$ such that $$|p_{s_is_j}^{(n)}-\pi_{s_j}|<\widehat{C}_{s_is_j}\widehat{\theta}^n$$
for some $\theta_2<\widehat{\theta}<1$. This completes the proof. 
\end{proof}

\begin{prop}\label{key2} Let $C_{s_is_j}$ be a positive constant depending only on $s_i$ and $s_j$. If there is a finite subset $B_N=\{b_0,b_1,\cdots,b_N\}\subset \mathcal{S}$ of alphabets such that $p_{s_is_j}^{(n),B_N}\le C_{s_is_j}\tau^n$ for some $0<\tau<1$, then the Markov chain $(\mathcal{S},p_{ij})$ converges to the stationary distribution exponentially.
\end{prop}
\begin{proof}
Since
\begin{align}\label{eq2}p_{s_is_j}^{(n),B_{N-1}}=p_{s_is_j}^{(n),B_N}+\sum_{r=1}^{n-1}f_{s_ib_N}^{(r),B_{N-1}}p_{b_Ns_j}^{(n-r),B_{N-1}},\end{align}
setting $s_i=b_N$, it follows that $$P_{b_Ns_j}^{B_{N-1}}(z)=\frac{P_{b_Ns_j}^{B_{N}}(z)}{1-F_{b_Nb_N}^{B_{N-1}}(z)}.$$ 
Thus, $P_{b_Ns_j}^{B_{N-1}}(z)$ is analytic in the disk $C_R$ for some $R>1$ and we obtain the estimate $p_{b_Ns_j}^{(n),B_{N-1}}<c_{s_j}\widetilde{\tau}^n$ for some $\widetilde{\tau}<1$.
Moreover, $f_{s_ib_N}^{(n),B_{N-1}}= p_{s_ib_N}^{(n),B_N}\le C_{s_ib_N}\tau^n$.
Now from (\ref{eq2}), we get $$p_{s_is_j}^{(n),B_{N-1}}\le \widehat{C}_{s_is_j}\widehat{\tau}^n.$$ for some $\widehat{C}_{s_is_j}>0$ and $\widetilde{\tau}<\widehat{\tau}<1$
Using induction for the sets $B_{N-1},B_{N-2},\cdots$, it finally gives us 
$$p_{s_is_j}^{(n),\{b_0\}}\le C'_{s_is_j}\theta^n$$
for some $\tau<\theta<1$. Now we can apply the previous Lemma.
\end{proof}

Now we state and prove the main theorem of this section.

\begin{thm} Let $\mathcal{T}$ be a locally finite tree. Suppose that $(\Gamma,\widetilde{F})$ is a pair of a full discrete subgroup $\Gamma<\textrm{Aut}(\mathcal{T})$ and a potential $\widetilde{F}$ for $\Gamma$ with $|m_{\Gamma,F}^{\nu^-,\nu^+}|<\infty$. Let us fix a vertex $o\in V\mathcal{T}$. If $L_\Gamma=k\mathbb{Z}$ and $(\Gamma,\widetilde{F})$ has WSG property, then for any $f,g\in C_c(\Gamma\backslash S_o^k\mathcal{T})$, we have
$$\left|\int_{\Gamma\backslash S_o^k\mathcal{T}} (f\circ \phi^{\circ kn})\cdot g\, dm_{\Gamma,F}^{\nu^-,\nu^+}-\int_{\Gamma\backslash S_o^k\mathcal{T}} f\,dm_{\Gamma,F}^{\nu^-,\nu^+}\int_{\Gamma\backslash S_o^k\mathcal{T}} g \,dm_{\Gamma,F}^{\nu^-,\nu^+}\right|=O(\theta^n)$$
for some constant $0<\theta<1$. The implied constant \color{black} depends only on $f$ and $g$.

\end{thm}

\begin{proof}
The dynamical system $(\Gamma\backslash S\mathcal{T},\phi,m_{\Gamma,F}^{\nu^-,\nu^+})$ is isomorphic to $(X_{(A,i)},\sigma,\lambda_{F})$. Let $Z_n$ be the Markov chain with the data $(\mathcal{S},p_{ij},\pi_j)$, where $\mathcal{S}=E(\Gamma\backslash S\mathcal{T})$, $p_{ij}=\frac{\lambda_{F}([s_i,s_j])}{\lambda_{F}([s_i])}$ and $\pi_j=\lambda_{F}([s_j])$. 
If $L_\Gamma=k\mathbb{Z}$, then $(\Gamma\backslash S_o^k\mathcal{T},\phi^{\circ k},m_{\Gamma,F}^{\nu^-,\nu^+})$ is mixing. Therefore, if we let $Z_{kn}$ be the $k$-step Markov chain on the equivalence of periodic classes in $(\mathcal{S},p_{ij},\pi_j)$, then it is irreducible and aperiodic.
Since $(\Gamma,\widetilde{F})$ has WSG property, by Lemma \ref{lemma} there is a finite subset $B$ of $\mathcal{S}$ such that $p_{ij}^{(kn),B}\le t_it_j^{-1}\rho^{kn}$ for some $0<\rho<1$. It follows that there exists $C_{ij},\theta>0$ such that $|p_{ij}^{(kn)}-\pi_j|<C_{ij}\theta^{kn}$ by Proposition \ref{key2}. Now Proposition \ref{correlation} completes the proof.
\end{proof}

\subsection{Examples of $\Gamma<\textrm{Aut}(\mathcal{T})$ with WSG property}

First, we state and prove the fact about the largeness of the critical exponent $\delta_\Gamma$ of $\Gamma$ when $\Gamma\backslash\backslash \mathcal{T}$ has a cuspidal ray. This is due to the finiteness of the normalized Patterson density.

\begin{prop} Let $\mathcal{T}$ be a locally finite tree. Let $(\Gamma,\widetilde{F})$ be a pair of non-elementary full discrete subgroup $\Gamma<\textrm{Aut}(\mathcal{T})$ and a potential $\widetilde{F}$ for $\Gamma$ with $\left|m_{\Gamma,F}^{\nu^-,\nu^+}\right|<\infty$. If $\Gamma\backslash\backslash\mathcal{T}$ has at least one cuspidal ray (see Definition \ref{cuspidal}) with edges $e_1,e_2,e_3,\cdots$ and indices $r_1,r_2,r_3,\cdots$, then $$\delta_{\Gamma,F}\ge\frac{1}{2}\log\overline{\lim}_{n\to\infty} (c_n^{1/n})$$
where $c_n=(r_n-1)r_{n-1}r_{n-2}\cdots r_1e^{F(e_1)+F(\overline{e_1})+\cdots F(e_n)+F(\overline{e_n})}.$ Furthermore, if the tree $\mathcal{T}$ is regular or bi-regular and $F$ is constant, then the inequality is strict.
\end{prop}
\begin{figure}[H]
\centering\includegraphics[width=0.8\linewidth]{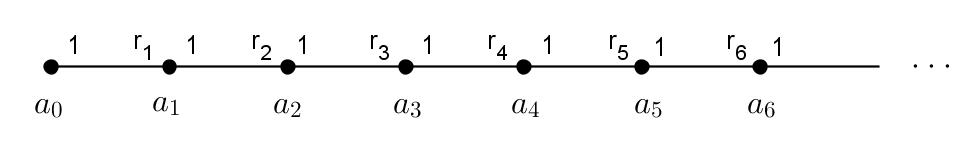}
\caption{Cuspidal ray with index $r_1,r_2,r_3,\cdots$}
\label{}
\end{figure}
\begin{proof}
Let $e_i$ be the edge with $\partial_0e_i=a_{i-1}$ and $\partial_1e_i=a_{i}$. Let $\nu^\pm$ be the Patterson density for $(\Gamma,\widetilde{F}^\pm)$ defined in Subsection \ref{density} and let $x=\nu^+_{a_0}(\mathcal{O}(\overline{e_1}))$. Since $\Gamma$ is non-elementary, $\nu^+$ has no atoms (cf. \cite{Bo}) and it follows that $x\ne 0$. Let $G\colon E\mathcal{T}\to\mathbb{R}$ be a $\Gamma$-invariant function given by $G(e)=F(e)+F(\overline{e})$. Because of the countably additivity and conformal property of $\nu^+$, we have
\begin{align*}
\nu^+_{a_0}(\mathcal{O}(e_1))
&=(r_1-1)\nu^+_{a_1}(\mathcal{O}(\overline{e_1}))e^{F(e_1)-\delta}+(r_2-1)r_1\nu^+_{a_2}(\mathcal{O}(\overline{e_1}))e^{F(e_1)+F(e_2)-2\delta}+\cdots \\
&=(r_1-1)xe^{G(e_1)-2\delta}+(r_2-1)r_1xe^{G(e_1)+G(e_2)-4\delta}+\cdots \\
&=\sum_{n=1}^{\infty}c_nxe^{-2n\delta}<\infty 
\end{align*}
which implies $\delta_{\Gamma,F}\ge\frac{1}{2}\log\overline{\lim}_{n\to\infty} (c_n^{1/n})$. If we assume further that $\mathcal{T}$ is a $(r+1,s+1)$ bi-regular tree $\mathcal{T}_{r+1,s+1}$ and $F$ is constant, then $$\nu^+_{a_0}(\mathcal{O}(e_1))= \left[(r-1)xe^{2(F-\delta)}+(s-1)rxe^{4(F-\delta)}\right]\sum_{n=1}^{\infty}(rs)^ne^{4n(F-\delta)},$$
so we obtain $\delta_{\Gamma,F}>F+\frac{1}{2}\log\sqrt{rs}.$
\end{proof}

In the rest of this subsection, we assume that the given potential $\widetilde{F}\colon E\mathcal{T}\to\mathbb{R}$ for $\Gamma<\textrm{Aut}(\mathcal{T})$ is constant. We give some examples of groups $\Gamma$ which have WSG property with these potentials.
\begin{exam} Let $\mathcal{T}_{r+1,s+1}$ be a $(r+1,s+1)$ bi-regular tree and let $\Gamma<\textrm{Aut}(\mathcal{T})$ be a geometrically finite full discrete subgroup with critical exponent $\delta_{\Gamma}$ for the zero potential. Then we have a decomposition $V(\Gamma\backslash\mathcal{T})=VD\cup VC_1\cup\cdots \cup VC_k\cup VF_1\cup\cdots\cup VF_l$ of the set of vertices as in propostion (\ref{split}) and there is a finite set $B\subseteq E(\Gamma\backslash\mathcal{T})$ such that $E(\Gamma\backslash\mathcal{T})=EB\cup EC_1\cup\cdots\cup EC_k\cup EF_1\cup\cdots EF_l$.

\begin{figure}[H]
\centering\includegraphics[width=0.8\linewidth]{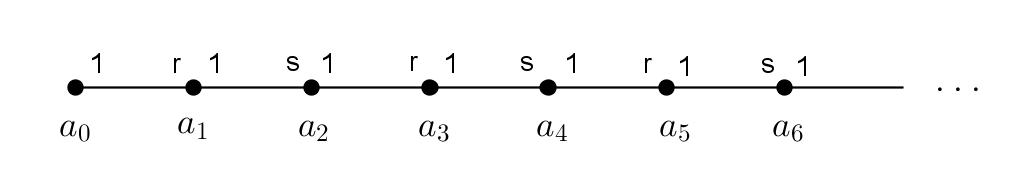}
\caption{Cuspidal ray of $\mathcal{T}_{r+1,s+1}$}
\label{tree}
\end{figure}

Let $\widetilde{F}$ be a given constant potential. Note that $F-\delta_{\Gamma,F}=-\delta_\Gamma$. For a cuspidal ray with the sequence of vertices $a_0,a_1,a_2,\cdots$, if we let $e_i$ be the oriented edge with $\partial_0e_i=a_{i-1}$ and $\partial_1e_i=a_i$, then 
$$p_{e_{2i-1}e_{2i}}=\frac{(s-1)re^{-2\delta_\Gamma}+(r-1)rse^{-4\delta_\Gamma}}{(r-1)+(s-1)re^{-2\delta_\Gamma}}<1\,,$$ $$p_{e_{2i}e_{2i+1}}=\frac{(r-1)se^{-2\delta_\Gamma}+(s-1)rse^{-4\delta_\Gamma}}{(s-1)+(r-1)se^{-2\delta_\Gamma}}<1\,,$$
$$p_{e_{i}\overline{e_{i}}}=1-p_{e_ie_{i+1}},$$ and $$p_{\overline{e_{i+1}}\,\overline{e_i}}=1.$$
Let $p=p_{e_{2i-1}e_{2i}}$ and $q=p_{e_{2i}e_{2i+1}}$. Define $t\colon E\mathcal{T}\to\mathbb{R}_{\ge 0}$ as follows: Choose $1<R<\sqrt[4]{pq}$ and $t_1$ sufficiently large. Let $t(\overline{e}_i)=R^{i}$ and $t(e_i)=t_i$ with
\begin{align*}
&t_{2i}=\frac{1}{p^iq^{i-1}R^{2i-1}}\left(t_1-(1-p)R^2\sum_{k=1}^i (pqR^4)^k-(1-q)pR^4\sum_{k=1}^{i-1} (pqR^4)^k\right) \\ &t_{2i+1}=\frac{1}{p^iq^{i}R^{2i}}\left(t_1-(1-p)R^2\sum_{k=1}^i (pqR^4)^k-(1-q)pR^4\sum_{k=1}^{i} (pqR^4)^k\right)
\end{align*} 
for $i\ge 1$. Then, we have $$pt(e_{2i})+(1-p)t(\overline{e_{2i-1}})=\frac{t(e_{2i-1})}{R},\,\, qt(e_{2i+1})+(1-q)t(\overline{e_{2i}})=\frac{t(e_{2i})}{R}\textrm{ and }t(\overline{e_{i-1}})=\frac{t(\overline{e_{i}})}{R},$$
so $\Gamma$ has WSG property with $(t,B,R^{-1})$.\hfill$////$

\end{exam}

Generalizing geometrically finite groups, we consider $\Gamma<\textrm{Aut}(\mathcal{T})$ for which a quotient graph of groups $\Gamma\backslash\backslash\mathcal{T}$ is a union of a finite graph of groups, finite rays and finite funnels. Consider the following ray (Firgure \ref{ray}). Let $e_i$ be the edge with $\partial_0e_i=a_{i-1}$ and $\partial_1e_i=a_{i}$.
Then, $$p_{e_ie_{i+1}}=\frac{|\Gamma_{\xi_i}|\nu^+_{a_i}(\mathcal{O}_{e_{i+1}})}{|\Gamma_{\xi_i,\xi_{i+1}}|\nu^+_{a_i}(\mathcal{O}_{e_{i}})}=\frac{s_{i}\cdot\nu^+_{a_i}(\mathcal{O}_{e_{i+1}})}{\nu^+_{a_i}(\mathcal{O}_{e_{i}})}.$$

\begin{figure}[H]
\centering\includegraphics[width=0.8\linewidth]{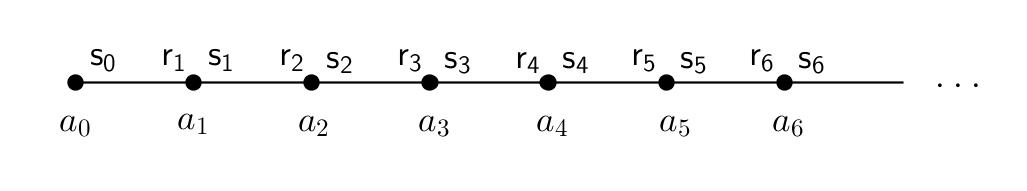}
\caption{General ray}
\label{ray}
\end{figure}

\begin{exam}
First, consider the case $s_i=2$ and $r_i=q-1$. Let $\Gamma<\textrm{Aut}(\mathcal{T})$ be a full discrete subgroup such that the decomposition of quotient graph of groups $\Gamma\backslash\backslash\mathcal{T}$ consists of a finite graph of groups $D$, finite funnels, and finite rays of type $(2,q-1)$. 
This is a quotient of a $(q+1)$-regular tree, but not geometrically finite. There is $\alpha>0$ such that 
$$\alpha<p_{e_i,e_{i+1}}<1-\alpha\textrm{ and }\alpha<p_{\overline{e_{i+1}},\overline{e_{i}}}<1-\alpha.$$

Choose a sequence $t_i\ge 1$ satisfying the following conditions: there exists $0<\rho<1$ such that for any $i=0,1,2,\cdots$,
\begin{align*}
(1) \,\, & p_{e_i,e_{i+1}}t_{i+1}+(1-p_{e_i,e_{i+1}})t_{i-1}\le t_i\rho; \textrm{ and} \\
(2) \,\, & p_{\overline{e_{i+1}},\overline{e_i}}t_{i-1}+(1-p_{\overline{e_{i+1}},\overline{e_i}})t_{i+1}\le t_i\rho.
\end{align*} 
Then, $\Gamma$ has WSG property with $(t,ED,\rho)$. \hfill$////$
\end{exam}

\begin{figure}[H]
\centering\includegraphics[width=0.75\linewidth]{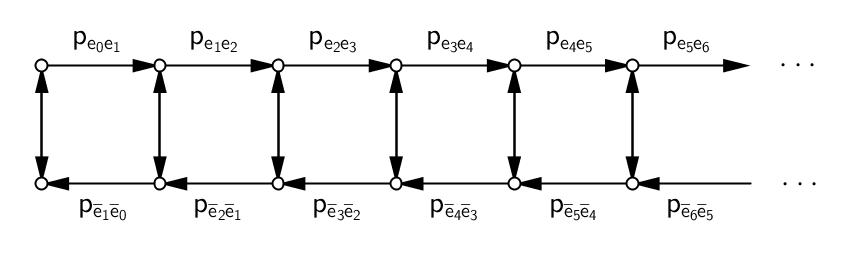}
\caption{Markov chain associated to the ray}
\label{planar}
\end{figure}

\begin{exam}
Now we consider the case that the ray of groups itself is not even an \emph{expander diagram}. The most famous example is appeared in \cite{BeLu}. In this example, for each $i$, either $r_i=1$ or $q$ and $s_i=q+1-r_i$. Although the regular representation of $\textrm{Aut}(\mathcal{T})$ into $L^2(\Gamma\backslash \textrm{Aut}(\mathcal{T}))$ has no spectral gap (see \cite{BeLu}), we have the exponential mixing property of the geodesic 2-translation map $\phi^{\circ 2}$. Indeed, as in the previous example, we can choose a sequence $t_i$ for $i\ge 0$ such that
\begin{align*}
(1) \,\, & p_{e_i,e_{i+1}}t_{i+1}+(1-p_{e_i,e_{i+1}})t_{i-1}\le t_i\rho; \textrm{ and} \\
(2) \,\, & p_{\overline{e_{i+1}},\overline{e_i}}t_{i-1}+(1-p_{\overline{e_{i+1}},\overline{e_i}})t_{i+1}\le t_i\rho.
\end{align*} 
Now we can apply Lemma (\ref{lemma}) and Proposition (\ref{key2}). \hfill$////$

\end{exam}


\section{Decay of matrix coefficient of algebraic groups and effective counting of discrete points}\label{five}

Before we proceed further, let us summarize here the results of the previous sections. Let us assume that a pair $(\Gamma,\widetilde{F})$ of a full discrete subgroup $\Gamma<\textrm{Aut}(\mathcal{T})$ and a potential $\widetilde{F}$ for $\Gamma$ has WSG property; for instance, let $\Gamma$ be a geometrically finite discrete subgroup with given a constant potential. If we assume that the length spectrum of $\Gamma$ is not arithmetic, then the geodesic translation map is \emph{exponentially} mixing. This is not a strong assumption since if we consider the restricted dynamical system $(\Gamma\backslash S_o^k\mathcal{T},\phi^{\circ k},m_{\Gamma,F}^{\nu^-,\nu^+})$ for $k=|\mathbb{Z}/L_\Gamma|$, then this is always true.
More precise statement is the following:

\bigskip
Let $\mathcal{T}$ be an arbitrary locally finite tree. If $(\Gamma,\widetilde{F})$ has WSG property and the length spectrum of $\Gamma$ is not arithmetic, then for any $f,g\in C_c(\Gamma\backslash S\mathcal{T})$, there is a constant $\kappa_1>0$ such that
$$\int _{\Gamma\backslash S\mathcal{T}}(f\circ\phi_\Gamma^{\circ n})\cdot gdm_{\Gamma,F}^{\nu^-,\nu^+}=\frac{1}{|m_{\Gamma,F}^{\nu^-,\nu^+}|}m_{\Gamma,F}^{\nu^-,\nu^+}(f)m_{\Gamma,F}^{\nu^-,\nu^+}(g)+O(e^{-\kappa_1 n}).$$

In this section, we focus on the case when $G<\textrm{Aut}(\mathcal{T})$ is an $\mathbb{F}$-points of an algebraic group $\underline{G}$ over a local field $\mathbb{F}$ and $\Gamma$ is a non-elementary discrete subgroup of $G$. Thus, $\mathcal{T}$ is a $(q^d+1,q^{d'}+1)$-bi-regular tree. In this case the geodesic 2-translation map $\phi^{\circ 2}$ is exponentially mixing on $\Gamma\backslash S_o^2\mathcal{T}$ for every $o\in V\mathcal{T}$. We want to translate the exponential mixing of $\phi_\Gamma^{\circ 2}$ with respect to $m_\Gamma^{\textrm{BM}}$ into the effective decay of matrix coefficient of the regular representation of $G$ on $L^2(\Gamma\backslash G)$. Furthermore, this allows us to count discrete orbit points in the Bruhat-Tits tree $\mathcal{T}$ of $G$ effectively.

\subsection{Equidistribution of $\phi_*^{\circ n}\nu_\mathcal{H}^{\textrm{Leb}}$}

In \cite{OS}, the authors gave the series of following propositions in the case of hyperbolic space $\mathbb{H}^n$. For the effective version of these, we refer \cite{MO} and \cite{OW}. Following directly their arguments, we get analogous results about the effective equidistribution of expanding horospheres in $\mathcal{T}$. 

We assume throughout this section that $L_\Gamma=2\mathbb{Z}$ and the given potential $\widetilde{F}$ is constant. 
Let $$h_\mathcal{T}=\lim_{n\to \infty} \frac{\log \#\{x\in V\mathcal{T}\,|\,d(x,o)\le n\}}{n}$$ be the \emph{volume entropy} (which does not depend on the chioice of $o\in V\mathcal{T}$) of $\mathcal{T}$. Let $\{\nu_x\}$ and $\{m_x\}$ be the $\delta_\Gamma$-dimensional Patterson-Sullivan density and the $h_\mathcal{T}$-dimensional Lebesgue density, respectively. We denote by 
\begin{align}\label{measures}
m_\Gamma^{\textrm{BM}}=m^{\nu,\nu},\quad m_\Gamma^{\textrm{BR}}=m^{m,\nu},\quad m^{\textrm{BR}}_{\ast,\Gamma}=m^{\nu,m}\quad\textrm{and}\quad m^{\textrm{Haar}}=m^{m,m}.
\end{align}
For each unstable horosphere $\widetilde{\mathcal{H}}$ in $S\mathcal{T}$, we consider the following locally finite Borel measures on $\widetilde{\mathcal{H}}$:
$$d\mu_{\widetilde{\mathcal{H}}}^{\textrm{Leb}}(\xi)=e^{h_\mathcal{T}\beta_{\xi^+}(o,\xi_0)}dm_o(\xi^+),\quad d\mu_{\widetilde{\mathcal{H}}}^{\textrm{PS}}(\xi)=e^{\delta_\Gamma\beta_{\xi^+}(o,\xi_0)}d\nu_o(\xi^+)$$ for some $o\in V\mathcal{T}$.
The measures $\mu_{\widetilde{\mathcal{H}}}^{\textrm{Leb}}$ and $\mu_{\widetilde{\mathcal{H}}}^{\textrm{PS}}$ are invariant under $\Gamma_{\widetilde{\mathcal{H}}}=\textrm{Stab}_\Gamma (\widetilde{\mathcal{H}})$ and hence induce measures on $\mathcal{H}=\Gamma_{\widetilde{\mathcal{H}}}\backslash\widetilde{\mathcal{H}}$.


The following proposition gives the equidistribution of the measure $\mu_{\widetilde{\mathcal{H}}}^{\textrm{PS}}$ toward $m_\Gamma^{\textrm{BM}}$.

\begin{prop} For any $f\in C_c(\Gamma\backslash S\mathcal{T})$ and $g\in L^1(\mathcal{H},\nu_{\mathcal{H}}^{\textrm{PS}})$, there is a constant $\kappa_2>0$ such that
$$\int_\mathcal{H} (f\circ\phi_\Gamma^{\circ 2n})\cdot gd\mu_\mathcal{H}^{\textrm{PS}}=\frac{\mu_\mathcal{H}^{\textrm{PS}}(g)}{|m^{\textrm{BM}}_\Gamma|}m^{\textrm{BM}}(f)+O(e^{-\kappa_2n})$$
with the implied constant depends only on $f$ and $g$.
\end{prop}
\begin{remark}\label{skinning} In general, when the potential $F$ is not constant, this statement is still true. We can also consider the generalized \emph{skinning measure} $d\sigma_{\widetilde{\mathcal{H}}}(\xi)=e^{C_{F,\xi^+}(o,\xi_0)}d\nu_o^+(\xi^+)$ and this equidistributes towards $m_{\Gamma,F}^{\nu^-,\nu^+}$ when pushed by the geodesic translation map (\cite{BPP}).
\end{remark}
\begin{thm} Let $f\in C_c(\Gamma\backslash S\mathcal{T})$ and $g\in C_c(\mathcal{H})$. Then there is a constant $\kappa_3>0$ such that
$$e^{(h_\mathcal{T}-\delta_\Gamma)n}\int_\mathcal{H} (f\circ\phi_\Gamma^{\circ 2n})\cdot gd\mu_\mathcal{H}^{\textrm{Leb}}=\frac{\mu_\mathcal{H}^{\textrm{PS}}(g)}{|m^{\textrm{BM}}_\Gamma|}m^{\textrm{BR}}(f)+O(e^{-\kappa_3n}).$$
The implied constant depends only on $f$ and $g$.
\end{thm}
\begin{thm} Let $f,g\in C_c(\Gamma\backslash S\mathcal{T})$. Then there is a constant $\kappa_4>0$ such that
$$e^{(h_\mathcal{T}-\delta_\Gamma)n} \int_{\Gamma\backslash S\mathcal{T}} f\cdot(g\circ \phi_\Gamma^{\circ 2n}) dm_\Gamma^{\textrm{Haar}}= \frac{1}{|m_\Gamma^{\textrm{BM}}|}\int f dm_{*,\Gamma}^{\textrm{BR}}\int g dm_\Gamma^{\textrm{BR}}+O(e^{-\kappa_4n})$$
with the implied constant depends only on $f$ and $g$.
\end{thm}

\begin{coro} Let $\underline{G}$ be a semi-simple $\mathbb{F}$-rank 1 algebraic group over $\mathbb{F}$ and $G=\underline{G}(\mathbb{F})$. Suppose that the Bruhat-Tits tree $\mathcal{T}$ of $G$ is $(q^d+1,q^{d'}+1)$-biregular and $\Gamma$ is a discrete subgroup of $G$ with $|m_{\Gamma}^{\textrm{BM}}|$ $<\infty$. Let $M<G$ be the stabilizer of a fixed bi-infinite geodesic line $\xi\in S\mathcal{T}$. Then there exists $\kappa>0$ such that for any $\Phi_1,\Phi_2\in C_c(\Gamma\backslash G/M)$, we have
$$\frac{q^{(d+d')j}}{e^{2\delta_\Gamma j}}\langle a_{2j}\cdot\Phi_1,\Phi_2\rangle= \frac{m^{\textrm{BR}}_\Gamma(\Phi_1)\cdot m^{\textrm{BR}}_{\ast,\Gamma} (\Phi_2)}{|m_{\Gamma}^{\textrm{BM}}|}+O(e^{-\kappa j}).$$ 
\label{Roblin2}
\end{coro}

\subsection{More on Burger-Roblin measure $m_\Gamma^\textrm{BR}$}\label{br}
We recall the definition of $\widetilde{m}_\Gamma^{\textrm{BR}}$ in (\ref{measures}). If $\psi\in C_c(S\mathcal{T})$, then $$\displaystyle\widetilde{m}^{BR}_\Gamma(\psi)=\int_{(\xi^+,\xi^-,s)\in S\mathcal{T}}\psi(g)e^{h_\mathcal{T}\beta_{g\cdot\xi^+}(o,g\cdot o)}e^{\delta\beta_{g\cdot\xi^-}(o,g\cdot o)}dm_o(g\xi^+)d\nu_o(g\xi^-)ds$$ and $$\displaystyle\widetilde{m}_{*,\Gamma}^{BR}(\psi)=\int_{(\xi^+,\xi^-,s)\in S\mathcal{T}}\psi(g)e^{\delta\beta_{g\cdot\xi^+}(o,g\cdot o)}e^{h_\mathcal{T}\beta_{g\cdot\xi^-}(o,g\cdot o)}d\nu_o(g\xi^+)dm_o(g\xi^-)ds.$$

Using Iwasawa decomposition, we may write $g=k_1p=k_2q$ for some $k_1,k_2\in K$, $p=a_{2j}k'\in P^+$ and $q=q_{2l}k''\in P^-$; for instance when $G=SL_2(\mathbb{F}),\,$ we can think of them as $
p=a_{2j}k'=\left( \begin{array}{cc}
\pi^{\,j} & b \\
0 & \pi^{-j} 
\end{array} \right)\in P^+$ and
$q=a_{2l}k''=\left( \begin{array}{cc}
\pi^{\,l} & 0 \\
c & \pi^{-l} 
\end{array} \right)\in P^-$ for some $b,c\in\mathbb{F}$. 
Thus, we have $$\beta_{g\cdot\xi^+}(o,g\cdot o)=\beta_{k_1p\cdot\xi^+}(o,k_1p\cdot o)=\beta_{\xi^+}(o,p\cdot o)=-2j$$ and
$$\beta_{g\cdot\xi^-}(o,g\cdot o)=\beta_{k_2q\cdot\xi^-}(o,k_2q\cdot o)=\beta_{\xi^-}(o,q\cdot o)=2l.$$

Note also that if we define the measure $d\lambda_g(n_x)$ on $N^+$ by $e^{h_\mathcal{T}\beta_{gn_x\cdot\xi^+}(o,gn_x\cdot o)}dm_o(gn_x\xi^+)$, then $d\lambda_g(n_x)=e^{h_\mathcal{T}\beta_{n_x\cdot\xi^+}(g^{-1}\cdot o,n_x\cdot o)}e^{h_\mathcal{T}\beta_{n_x\cdot\xi^+}(o,g^{-1}\cdot o)}dm_o(g\xi^+)=d\lambda_e(n_x)$. Since $d\lambda_e(gn_x)=d\lambda_g(n_x)=d\lambda_e(n_x)$, it follows that $\lambda_e$ is a Haar meausre on $N^+$. We will denote it by $dn=d\lambda_e(n)$ as an integral on $N^+$ (see \cite{OS}, Lemma 7.1).

Therefore, we have 
$$\widetilde{m}^{\mathrm{BR}}_\Gamma(\psi)=\displaystyle\sum_{l\in\mathbb{Z}}\int_{k_2\in K}\int_{n^-\in N^-/N^-\cap K}\psi(k_2a_ln^-)e^{2\delta l}dn^-d\nu_o(k_2\xi^-)$$ 
and similarly 
$$\widetilde{m}_{*,\Gamma}^{\mathrm{BR}}(\psi)=\sum_{j\in\mathbb{Z}}\int_{k_1\in K}\int_{n\in N^+/N^+\cap K}\psi(k_1a_jn)e^{-2\delta j}dnd\nu_o(k_1\xi^+).$$ 

\subsection{Definition of the measure $\mathcal{M}_G^{(\Gamma)}$}

Let $G=\underline{G}(\mathbb{F})$ be the $\mathbb{F}$-rational points of a semi-simple rank one algebraic group over $\mathbb{F}$ whose Bruhat-Tits tree $\cal{T}$ is a $(q^d+1,q^{d'}+1)$-biregular tree. We have the natural action of $G$ on $\cal{T}$. Let $K$ be the stabilizer of some fixed vertex $o$ in $\cal{T}$. Fix a geodesic $\xi\in S\mathcal{T}$ whose base point is $o$ and let $o'$ be the unique vertex in $\mathcal{H}_{\xi}^+(1)$ for which $d_\mathcal{T}(o,o')=1$. Let $K'$ be the stabilizer of the vertex $o'$. Then $K$ and $K'$ are maximal compact subgroups of $H$ and they are not conjugate to each other. In this case, we can identify $V\mathcal{T}\simeq G/K\cup G/K'$.

For a given measurable subset $B_R$ of $G$, we define the \emph{counting function} $F_R$ on $\Gamma\backslash G\times\Gamma\backslash G$ by
\begin{gather*}
\begin{aligned}
F_R(g,h) &=\sum_{\gamma\in\Gamma}\chi_{B_R}(g^{-1}\gamma h).
\end{aligned}
\end{gather*}

Then it follows that $F_R(e,g)=\left|\Gamma g\cap B_R\right|$. For any real-valued function $\Phi\in C_c(\Gamma\backslash G)$ and $k\in K$, we set $\Phi^k(g)\!:=\!\Phi(gk)$ and $\Phi_k(g)\!:=\!\Phi(kg)$.

\begin{lem} Suppose that $o$ is a vertex of degree $q^d+1$ in $\cal{T}$. If $f\in C_c(G)$, then we have
$$\int_{g=k_1a_{2j}k_2\in G}f(g)dg=\int_{k_1\in K}\sum_{a_{2j}\in Z^+}\int_{k_2\in K} f(k_1a_{2j}k_2)\Delta(2j)dk_2dk_1$$ where $\Delta(0)=1$ and $\Delta(2j)=(q^d+1)q^{(d+d')j-d}$ for $j\ge 1$, and $dk_1$ and $dk_2$ denote the probability measures on $K$.
\end{lem}
\begin{proof}
Note that $\Delta(2j),(j\ge0)$ is the number of vertices of a sphere of radius $2j$ (centered at $x_0$) in a $(q^d+1,q^{d'}+1)$-regular tree. Let $f^K$ be the function on $G/K$ given by $f^K(xK)=\int_{k\in K}f(xk)dk$. Let $W_0$ be a spherical Weyl group so that $K=BW_0B$ corresponds to a vertex $o\in\mathcal{T}$ and $dk$ is a probability measure on $K$. Now we have \begin{align*}
\displaystyle \sum_{a_{2j}\in Z^+}\int_{k_1}\int_{k_2}f(k_1a_{2j}k_2)\Delta(2j)dk_1dk_2 &=\sum_{a_{2j}\in Z^+}\int_{k_1}f^{K}(k_1a_{2j}K)\Delta(2j)dk_1 \\ &=\sum_{k_1a_{2j}K\in G/K}f^{K}(k_1a_{2j}K)
\end{align*}
Meanwhile, if we denote by $m_{G/K}$ the $G$-invariant measure on $G/K$, then
\begin{align*}
\int_{g=k_1a_{2j}k_2} f(g)dg&=\int_{k_1a_{2j}K\in G/K}\int_{k_2\in K} f(k_1a_{2j}k_2)\,dk_2\,dm_{G/K} \\
&=\int_{k_1a_{2j}K\in G/K}f^{K}(k_1a_{2j}K)\,dm_{G/K}.
\end{align*} This completes the proof since the counting measure $m$ on $V\mathcal{T}\simeq G/K\cup G/K'$ given by $m(E)=\#E$ is $G$-invariant.
\end{proof}

Now for any $\Phi_1,\Phi_2\in C_c(\Gamma\backslash G)$, using the previous lemma and following the argument of Duke-Rudnick-Sarnak (\cite{DRS}), we have
\label{a}
\begin{align*}
\langle F_R,\Phi_1\otimes\Phi_2\rangle _{\Gamma\backslash G\times\Gamma\backslash G} &= \int_{\Gamma\backslash G}\int_{\Gamma\backslash G} \sum_{\gamma\in\Gamma}\chi_{B_R}(x^{-1}\gamma y)\Phi_1(x)\Phi_2(y)dxdy \\ \displaybreak[1]
&= \int_{\Gamma\backslash G}\Phi_1(x)\int_{\Gamma\backslash G}\sum_{\gamma\in\Gamma}\chi_{B_R}(x^{-1}\gamma y)\Phi_2(\gamma y)dydx \\ \displaybreak[1]
&= \int_{\Gamma\backslash G}\Phi_1(x)\int_G \chi_{B_R}(x^{-1}g)\Phi_2(g)dgdx \\ \displaybreak[1]
&= \int_{g\in B_R}\int_{\Gamma\backslash G}\Phi_1(x)\Phi_2(xg)dxdg \numberthis\label{eqn}\\
&= \int_{g\in B_R}\langle \Phi_1,g\cdot\Phi_2\rangle_{L^2(\Gamma\backslash G)}dg \\ \displaybreak[1]
&= \int_{k_1a_{2j}k_2\in B_R}\left[\int_{\Gamma\backslash G}\Phi_1(xk_2^{-1})\Phi_2(xk_1a_{2j})dx\right]\cdot \Delta(2j)djdk_1dk_2 \\
&=\int_{k_1}\int_{k_2}\sum_{a_{2j}\in k_1^{-1}B_Rk_2^{-1}}\left(\Delta(2j)\cdot\langle\Phi_1^{k_2^{-1}},a_{2j}\cdot\Phi_2^{k_1}\rangle\right)dk_1dk_2.
\end{align*}

Let $\delta=\delta_\Gamma$ be the critical exponent of $\Gamma$, and let $\nu_o$ be the Patterson-Sullivan density based at the vertex $o\in \cal{T}$ (see Subsection \ref{density} for the definition). We define the measure $\mathcal{M}_G^{(\Gamma)}$ on $G$ by 
\begin{align*}
\mathcal{M}_G^{(\Gamma)}(\psi)&=\frac{1}{\vert m_\Gamma^{\textrm{BM}}\vert}\int_{k_1\in K}\int_{k_2\in K}\left(\sum_{j\in\mathbb{N}\cup\{0\}}\frac{e^{2\delta j}}{q^{(d+d')j}}\Delta(2j)\psi(k_1a_{2j}k_2)\right)d\nu_o(k_1\xi^-)d\nu_o(k_2\xi^+)\\
&=\frac{1}{\vert m_\Gamma^{\textrm{BM}}\vert}\int\!\!\!\!\int_{k_1,k_2\in K}\!\!\left(\psi(k_1k_2)+\frac{q^d+1}{q^d}\sum_{j\in\mathbb{N}}e^{2\delta j}\psi(k_1a_{2j}k_2)\right)d\nu_o(k_1\xi^-)d\nu_o(k_2\xi^+).
\end{align*}
for any $\psi\in C_c(G)$.
For example, if we let $Z_{2R}^+=\{a_{2j}\,\vert\, 0\le j\le R\}$, then we get
$$\mathcal{M}_G^{(\Gamma)}(KZ_{2R}^+K)=\left(1+\frac{q^d+1}{q^d}\cdot\frac{e^{2\delta}(e^{2\delta R}-1)}{e^{2\delta}-1}\right) \frac{1}{\vert m_\Gamma^{\textrm{BM}}\vert}.$$


\subsection{About well-rounded sets}
Using the exponential mixing property of the geodesic translation map $\phi^{\circ 2}\colon S\mathcal{T}\to S\mathcal{T}$, we count the number of points of a $\Gamma$-orbit in a well-rounded family of compact sets. Here we give a definition of well-roundedness following \cite{MO}. 
\begin{defn} A family of compact sets $\{B_R\}\subseteq G$ is called \emph{effectively well-rounded} with respect to $\Gamma$ if there exists $p>0$ such that for any small $\epsilon>0$ and sufficiently large $R$, we have $\mathcal{M}_G^{(\Gamma)}(B_{R,\epsilon}^+ -B_{R,\epsilon}^-)=O(\epsilon^p)\mathcal{M}_G^{(\Gamma)}(B_R).$
\end{defn}

\begin{defn}[\cite{MO}, Definition 6.10]\label{admissible}
We say a Borel subset $\Omega\subseteq K$ with $\nu_o(\Omega^{-1})>0$ is \emph{admissible} if there exists $0<s\le 1$ such that for all small $\epsilon>0$,
$$\nu_o\left(\Omega^{-1}K_{\epsilon}-\bigcap_{k\in K_\epsilon} \Omega^{-1}k\right)\ll\epsilon^s$$ with the implied constant depending only on $\Omega$. Here, $K_\epsilon$ is the $\epsilon$-neighborhood of identity in $K$ with respect to the given left-invariant metric.
\end{defn}

\begin{exam}[Admissible bisectors]
\label{well-roundness of bisectors}
Let $\Omega_1,\Omega_2\subset K$ be Borel subsets in $K$ such that $\Omega_1^{-1}$ and $\Omega_2$ are admissible. If we write $S_R(\Omega_1,\Omega_2)=\Omega_1 Z_{2R}^+\Omega_2$, then the family of compact sets $\{S_R(\Omega_1,\Omega_2)\colon R\in\mathbb{N}\}$ is effectively well-rounded with respect to $\Gamma$.
\end{exam}
\begin{proof}
First, by the definition of the measure $\mathcal{M}_G^{(\Gamma)}$, we have $$\displaystyle\mathcal{M}_G^{(\Gamma)}(S_R(\Omega_1,\Omega_2))=\frac{1}{|m_{\Gamma}^{\textrm{BM}}|}\nu_o(\Omega_1)\nu_o(\Omega_2^{-1})\sum_{0\le j\le R}\frac{\Delta(2j)}{q^{(d+d')j}}e^{2\delta j}.$$ 
Also for all $R\in\mathbb{N}$ and small $\epsilon>0$, we have
$$S_T(\Omega_1,\Omega_2)G_\epsilon\subset \Omega_{1}Z_{2R}^+\Omega^{-1}_{2}K_{\epsilon}$$ 
because $K$ is an open subgroup of $G$. (In fact, this implies the strong wavefront property of Cartan decomposition.)
Hence
\begin{equation*}
\begin{aligned}
\mathcal{M}_G^{(\Gamma)}(S_R(\Omega_1,\Omega_2)G_\epsilon) &\ll
\frac{1}{|m_{\Gamma}^{\textrm{BM}}|}\nu_o(\Omega_{1})\nu_o(\Omega^{-1}_{2}K_\epsilon)\sum_{0\le j\le R}\frac{\Delta(2j)}{q^{(d+d')j}}e^{2\delta j}\\ 
&\ll \frac{1}{|m_{\Gamma}^{\textrm{BM}}|}(1+O(\epsilon^s))\nu_o(\Omega_1)\nu_o(\Omega_2^{-1})\sum_{0\le j\le R}\frac{\Delta(2j)}{q^{(d+d')j}}e^{2\delta j} \\
&= (1+O(\epsilon^s))\mathcal{M}_G^{(\Gamma)}(S_R(\Omega_1,\Omega_2))
\end{aligned}
\end{equation*}
for some $s>0$. Similarly, we also have that 
$$\mathcal{M}_G^{(\Gamma)}\left(\bigcap_{g\in G_\epsilon}S_R(\Omega_1,\Omega_2)g\right)=(1-O(\epsilon^s))\mathcal{M}_G^{(\Gamma)}(S_R(\Omega_1,\Omega_2)) \textrm{ for some }s>0.$$ Therefore, the family $\{S_R(\Omega_1,\Omega_2)\colon R\in\mathbb{N}\}$ is effectively well-rounded with respect to $\Gamma$.
\end{proof}

Meanwhile, for a subset $E\in V\mathcal{T}$, the boundary $\partial E$ of $E$ is defined by the set $\{x\in V\mathcal{T}\,|\,x\notin E \textrm{ and } \exists\, y\in E\colon d(y,x)=1 \}$. We can prove that if the boundaries of a family of finite sets in $V\mathcal{T}$ are sufficiently small, then the family of finite sets itself is well-rounded. More precisely, 

\begin{prop}
Let $\{E_R\}$ be a family of finite subsets in $V\mathcal{T}\simeq G/K\cup G/K'$ satisfying the following property: $\vert E_R\vert\to\infty$ as $R\to\infty$ and there exists a universal constant $\beta>1$ such that $\displaystyle \frac{\vert\partial E_R\vert}{\vert E_R\vert}\le R^{-\beta}$ for every $R\ge 1$. Then $\{E_R\}$ is a well-rounded family.
\end{prop}
\begin{proof}
Let $D_R=E_R\cap (G\cdot o)$ and $D_R'=E_R\cap (G\cdot o')$.
Taking $U=Ka_2Ka_2^{-1}$, we have $$\displaystyle \frac{\mathcal{M}_G^{(\Gamma)}(D_RKU-\cap_{u\in U}D_RKu)}{\mathcal{M}_G^{(\Gamma)}(D_RK)}\le\frac{(q^d+1)\mathcal{M}_G^{(\Gamma)}((\partial D_R)K)}{\mathcal{M}_G^{(\Gamma)}(D_RK)}$$ and because of the condition and the definition of measure $\mathcal{M}_G^{(\Gamma)}$, it follows that $$\displaystyle\frac{(q^d+1)\mathcal{M}_G^{(\Gamma)}((\partial D_R)K)}{\mathcal{M}_G^{(\Gamma)}(D_RK)}\le(q^d+1)^2R^{-\beta}.$$ The similar argument gives us 
$$\frac{\mathcal{M}_G^{(\Gamma)}(D_R'K'V-\cap_{v\in V}D_R'K'v)}{\mathcal{M}_G^{(\Gamma)}(D_R'K')}\le (q^{d'}+1)^2R^{-\beta}$$ for $V=K'a_2K'a_2^{-1}$. 
These inequalities complete the proof.
\end{proof}

\subsection{Effective counting of discrete points in $G$}

\begin{thm} For an effectively well-rounded family $\{B_R\}$ of compact subsets of $G$, there is a constant $\eta>0$ for which we have
$$\vert\Gamma g\cap B_R\vert =\mathcal{M}_G^{(\Gamma)}(B_R)+O(\mathcal{M}_G^{(\Gamma)}(B_R)^{1-\eta}).$$
\end{thm}
\begin{proof} For an element $g=a_ln^-k_1\in ZN^{-}K$, define $\kappa^{-}(g)=k_1$ and $H^{-}(g)=l$. Similarly, if $g=a_{l'}n^+k_2\in ZN^+K$, then we define $\kappa^+(g)=k_2$ and $H^+(g)=l'$. We have an integral formula
$$\int_{g=a_ln^-k} f(g)\,dg=\int_{k\in K}\int_{n^-\in N^-/(N^-\cap K)}\sum_{\l=-\infty}^{\infty}f(a_ln^-k_1)\,dn^-dk_1.$$
Let $\Phi_\epsilon\in C_c(G)$ be a continuous function whose support is a compact subset of an $\epsilon$-neighborhood of $e$ in $G$. By equation \eqref{eqn} we have
\begin{align*} 
&\langle F_{B_{R,\epsilon}}^+,\Phi_\epsilon\otimes\Phi_\epsilon\rangle \\
=&\int_{g\in B_{R,\epsilon}^+}\langle\Phi_\epsilon,g\cdot\Phi_\epsilon\rangle_{L^2(\Gamma\backslash G)}dm^{\textrm{Haar}}(g) \\\displaybreak[0]
=&\int_{k_1a_{2j}k_2\in B_{R,\epsilon}}\langle\Phi^{k_1^{-1}}_\epsilon,a_{2j}\cdot\Phi^{k_2}_\epsilon\rangle_{L^2(\Gamma\backslash G)}dm^{\textrm{Haar}}(g).
\end{align*}
Now by Corollary \ref{Roblin2} this is asymptotically equivalent to 
\begin{align*}
&\frac{1}{\vert m_{\Gamma}^{\textrm{BM}}\vert}\int_{k_1a_{2j}k_2\in B_{R,\epsilon}}\frac{e^{2\delta j}}{q^{(d+d')j}}m^{\textrm{BR}}_\Gamma(\Phi^{k_1^{-1}}_\epsilon) m^{\textrm{BR}}_{*,\Gamma}(\Phi^{k_2}_\epsilon)\,dm^{Haar}(g) \\\displaybreak[0]
=&\frac{1}{\vert m_{\Gamma}^{\textrm{BM}}\vert}\int_{k_1\in K}\int_{k_2\in K}\sum_{a_{2j}\in k_1^{-1}B_{R,\epsilon}^+k_2^{-1}\cap Z^+}\frac{e^{2\delta j}\Delta(2j)}{q^{(d+d')j}}m^{\textrm{BR}}_\Gamma(\Phi^{k_1^{-1}}_\epsilon)m^{\textrm{BR}}_{*,\Gamma}(\Phi^{k_2}_\epsilon)\,dk_1dk_2.
\end{align*}
The formula of Burger-Roblin measure given in subection \ref{br} implies that this equation is equal to
\begin{align*}
&\frac{1}{|m_{\Gamma}^{\textrm{BM}}|} \int_{k_1\in K}\int_{k_2\in K}\sum_{a_{2j}\in k_1^{-1}B_{R,\epsilon}^+k_2^{-1}\cap Z^+}\frac{e^{2\delta j}\Delta(2j)}{q^{(d+d')j}} \\ \displaybreak[0] &\cdot\left[\int_{k\in K}\int_{n^-\in N^-/N^-\cap K}\sum_{l=-\infty}^{\infty}\left(\Phi_\epsilon(ka_ln^-k_1^{-1})e^{2\delta l}\right)dn^-d\nu_o(k\xi^-)\right] \\ &\cdot\left[\int_{k'\in K}\int_{n^+\in N^+/N^+\cap K}\sum_{l'=-\infty}^{\infty}\left(\Phi_\epsilon(k'a_{l'}n^+k_2)e^{-2\delta l'}\right)dn^+d\nu_o(k'\xi^+)\right]dk_1dk_2.
\end{align*}
Taking $g_1=a_ln^-k_1$ and $g_2=a_{l'}n^+k_2$ and by the well-roundedness of $\{B_R\}$ we can conclude that
\begin{align*}
&\langle F_{B_{R,\epsilon}}^+,\Phi_\epsilon\otimes\Phi_\epsilon\rangle \\ 
\sim\,&\frac{1}{\vert m_{\Gamma}^{\textrm{BM}}\vert}\int_{k\in K}\int_{k'\in K}\int_{g_1\in G}\int_{g_2\in G}\Phi_\epsilon(kg_1)e^{2\delta H^-(g_1)}\Phi_\epsilon(k'g_2)e^{-2\delta H^+(g_2)}\\ \displaybreak[0]
&\cdot\sum_{a_{2j}\in \kappa(g_1)^{-1}B_{R,\epsilon}^+\kappa(g_2)^{-1}\cap Z^+}\frac{e^{2\delta j}\Delta(2j)}{q^{(d+d')j}}\,dg_1dg_2d\nu_o(k\xi^-)d\nu_o(k'\xi^+) \\ \displaybreak[1]
=\,&\frac{1}{\vert m_{\Gamma}^{\textrm{BM}}\vert}\int_{k}\int_{k}\int_{g_1}\int_{g_2}\sum_{a_{2j}\in \kappa(k^{-1}g_1)^{-1}B_{R,\epsilon}^+\kappa(k'^{-1}g_2)^{-1}\cap Z^+}\frac{e^{2\delta j}\Delta(2j)}{q^{(d+d')j}}dg_1dg_2d\nu_o(k\xi^-)d\nu_o(k'\xi^+) \\
\le\,&\frac{1}{\vert m_{\Gamma}^{\textrm{BM}}\vert}\int_{k\in K}\int_{k'\in K}\sum_{a_{2j}\in k_1^{-1}B_{R,c\epsilon}^+k_2^{-1}\cap Z^+}\frac{e^{2\delta j}\Delta(2j)}{q^{(d+d')j}}d\nu_o(k\xi^-)d\nu_o(k'\xi^+) \\
=\,&\mathcal{M}_G^{(\Gamma)}(B_{R,c\epsilon}^+)=(1+O(\epsilon^p))\mathcal{M}_G^{(\Gamma)}(B_R).
\end{align*}
By the similar argument, we also have that $\langle F_{B_{R,\epsilon}}^-,\Phi_\epsilon\otimes\Phi_\epsilon\rangle=(1-O(\epsilon^p))\mathcal{M}_G^{(\Gamma)}(B_R).$
Now the inequalities
$\langle F_{B_{R,\epsilon}}^-,\Phi_\epsilon\otimes\Phi_\epsilon\rangle\le F_{B_R}(e,g)=\vert\Gamma g\cap B_R\vert\le\langle F_{B_{R,\epsilon}}^+,\Phi_\epsilon\otimes\Phi_\epsilon\rangle$ gives us the result.
\end{proof}

For a subset $\Omega\subset\partial_\infty\cal{T}$ and $x\in V\mathcal{T}$, we denote by $S_x(\Omega)\subset\cal{T}$ the set of all points lying in geodesics emanating from $x$ toward $\Omega$, and by $B(x,n)\subset V\mathcal{T}$ the ball of radius $n$ centered at $x$. 

\begin{coro}\label{maincoro} Let $x$ be a degree $q^d+1$ vertex of the Bruhat-Tits tree $\cal{T}$ of $G$ and $\Gamma$ be a discrete subgroup of $G$ with $|m^{\textrm{BM}}_\Gamma|<\infty$ and having WSG property. Then we have 
$$\#\{\gamma\in\Gamma\colon \gamma\cdot x\in S_x(\Omega)\cap B(x,2n)\}=\frac{e^{2\delta_\Gamma}(q^d+1)(e^{2\delta_\Gamma n}-1)|\Gamma_x|}{q^d(e^{2\delta_\Gamma}-1)\|m^{\textrm{BM}}_\Gamma\|}\nu_x(\Omega)+O(e^{(2\delta_\Gamma-\kappa)n}),$$
for some $\kappa>0$.
\end{coro}

Using arguments of \cite{PPS} and \cite{BPP}, one can obtain error rate on the number of edge paths of length at most $n$ in $\Gamma\backslash\backslash\mathcal{T}$ with weights $e^{\int_x^{\gamma y}\widetilde{F}}$ using the equidistribution of the skinning measure $d\sigma_{\widetilde{\mathcal{H}}}$ (see Remark \ref{skinning}). The following corollary gives the precise statement.

\begin{coro}[\cite{BPP}] Let $x$ be a degree $q^d+1$ vertex of the Bruhat-Tits tree $\mathcal{T}$ of $G$ and suppose that $(\Gamma,\widetilde{F})$ is a pair of discrete subgroup $\Gamma<G$ and a potential $\widetilde{F}$ for $\Gamma$ with WSG property. Let $\mathcal{N}_x(n)$ be the weighted counting function of the closed path in $\Gamma\backslash\backslash\mathcal{T}$ of length at most $n$ with base point $\Gamma x$. In other words, let
$$\mathcal{N}_x(n)=\sum_{\gamma\colon d_\mathcal{T}(\gamma x,x)\le n}e^{\int_x^{\gamma x}\widetilde{F}}.$$
Then as $n\to\infty$, we have
$$\mathcal{N}_x(2n)=\frac{e^{2\delta_{\Gamma,F}}\|\nu^-_x\|\|\nu^+_x\||\Gamma_x|}{(e^{2\delta_{\Gamma,F}}-1)\|m_{\Gamma,F}^{\nu^-,\nu^+}\|}e^{2n\delta_{\Gamma,F}}+O(e^{(2\delta_{\Gamma,F}-\kappa)n})$$
for some $\kappa>0$. 
\end{coro}



\end{document}